\documentclass[11pt, a4paper, oneside]{amsart}

\usepackage[english]{babel}
\usepackage{amsmath, amsthm, amsfonts, mathrsfs, amssymb}
\usepackage{mathtools}
\mathtoolsset{centercolon}
\usepackage{booktabs}
\usepackage[shortlabels]{enumitem}
\setlist[itemize]{leftmargin=20pt}
\usepackage[colorlinks]{hyperref}
\usepackage[hmargin=3cm,vmargin=3cm]{geometry}

\usepackage{color}
\usepackage{graphicx}

\newcommand{\N}{\ensuremath{\mathbb{N}}}
\newcommand{\Z}{\ensuremath{\mathbb{Z}}}

\newcommand{\R}{\ensuremath{\mathbb{R}}}
\newcommand{\C}{\ensuremath{\mathbb{C}}}

\newcommand{\mb}{\mathbf}
\newcommand{\mc}{\mathcal}
\newcommand{\ms}{\mathscr}

\DeclarePairedDelimiter\abs{\lvert}{\rvert}
\DeclarePairedDelimiter\brac[]
\DeclarePairedDelimiter\cbrace\{\}
\DeclarePairedDelimiter\ha()
\DeclarePairedDelimiter{\ip}\langle\rangle
\DeclarePairedDelimiter{\nrm}\lVert\rVert

\newcommand{\nrmb}[1]{\bigl\|#1\bigr\|}

\newcommand{\hab}[1]{\bigl(#1\bigr)}
\newcommand{\cbraceb}[1]{\bigl\{#1\bigr\}}

\newcommand{\bracb}[1]{\bigl[#1\bigr]}

\newcommand{\nrms}[1]{\Bigl\|#1\Bigr\|}

\newcommand{\has}[1]{\Bigl(#1\Bigr)}

\DeclareMathOperator{\re}{Re}

\DeclareMathOperator{\sgn}{sgn}

\DeclareMathOperator{\supp}{supp}
\DeclareMathOperator{\ind}{\mathbf{1}}
\DeclareMathOperator{\UMD}{UMD}

\DeclareMathOperator{\divv}{div}
\newcommand{\RH}{RH}

\DeclareMathOperator*{\essinf}{ess\,inf}
\DeclareMathOperator{\full}{full}
\DeclareMathOperator{\lac}{lac}
\DeclareMathOperator{\BHT}{BHT}
\DeclareMathOperator{\pv}{p.v.}

\newcommand{\sumkn}{\sum_{k= 1}^n}

\newcommand{\dd}{\hspace{2pt}\mathrm{d}}

\def\avint_#1{\mathchoice{\mathop{\kern 0.2em\vrule width 0.6em height 0.69678ex depth -0.58065ex \kern -0.8em \intop}\nolimits_{\kern -0.4em#1}}{\mathop{\kern 0.1em\vrule width 0.5em height 0.69678ex depth -0.60387ex \kern -0.6em \intop}\nolimits_{#1}} {\mathop{\kern 0.1em\vrule width 0.5em height 0.69678ex depth -0.60387ex \kern -0.6em \intop}\nolimits_{#1}} {\mathop{\kern 0.1em\vrule width 0.5em height 0.69678ex depth -0.60387ex \kern -0.6em \intop}\nolimits_{#1}}}
\newcommand{\inc}{\mb{\phi}}   

\newtheorem{theorem}{Theorem}

\newtheorem{lemma}[theorem]{Lemma}
\newtheorem{proposition}[theorem]{Proposition}

\theoremstyle{remark}
\newtheorem{remark}[theorem]{Remark}
\newtheorem{example}[theorem]{Example}

\theoremstyle{definition}
\newtheorem{definition}[theorem]{Definition}

\numberwithin{theorem}{section}
\numberwithin{equation}{section}
\title[Vector-valued extensions of operators through extrapolation]{Vector-valued extensions of operators through multilinear limited range extrapolation}
\author{Emiel Lorist and Zoe Nieraeth}
\thanks{The first author is supported by the VIDI subsidy 639.032.427 of the Netherlands Organisation for Scientific Research (NWO)}

\address{Delft Institute of Applied Mathematics \\ Delft University of Technology \\ P.O. Box 5031\\ 2600 GA Delft \\The Netherlands}
\email{e.lorist@tudelft.nl}
\email{znieraeth@bcamath.org}
\allowdisplaybreaks

\begin{document}
\begin{abstract}
We give an extension of Rubio de Francia's extrapolation theorem for functions taking values in $\UMD$ Banach function spaces to the multilinear limited range setting. In particular we show how boundedness of an $m$-(sub)linear operator
\[
T:L^{p_1}(w_1^{p_1})\times\cdots\times L^{p_m}(w_m^{p_m})\to L^p(w^p)
\]
for a certain class of Muckenhoupt weights yields an extension of the operator to Bochner spaces $L^{p}(w^p;X)$ for a wide class of Banach function spaces $X$, which includes certain Lebesgue, Lorentz and Orlicz spaces.

We apply the extrapolation result to various operators, which yields new vector-valued bounds. Our examples include the bilinear Hilbert transform, certain Fourier multipliers and various operators satisfying sparse domination results.
\end{abstract}

\keywords{Extrapolation, limited range, multilinear, UMD, Muckenhoupt weights, Banach function space, $p$-convexity, bilinear Hilbert transform, Fourier multipliers, sparse domination}

\subjclass[2010]{Primary: 42B25; Secondary: 42B15, 46E30}


\maketitle

\section{Introduction}

Scalar-valued extrapolation, using the theory of Muckenhoupt weights, has proven to be an essential tool in harmonic analysis. The classical extrapolation result (see \cite{Ru82} and \cite[Chapter IV]{GR85}) says that if  a (sub)linear operator $T$ satisfies for a fixed $p_0\in (1,\infty)$ and all weights $w$ in the \emph{Muckenhoupt} class $A_{p_0}$ the norm inequality
\begin{equation}\label{eq:extrapinput}
\|Tf\|_{L^{p_0}(w)}\leq C\,\|f\|_{L^{p_0}(w)}
\end{equation}
for all $f \in L^{p_0}(w)$, then we have for all $p\in (1,\infty)$ and all weights $w\in A_p$
\begin{equation}\label{eq:extrapout}
\|Tf\|_{L^p(w)}\leq C\,\|f\|_{L^p(w)}
\end{equation}
for all $f \in L^{p}(w)$. Numerous generalizations of this result have appeared, see for example \cite{AM07, CMP04, CM17, GM04, HMS88}. We mention several of them.

It was shown by Grafakos and Martell \cite{GM04} that extrapolation extends to the \emph{multilinear} setting. Indeed, they showed that given fixed exponents $p_1,\ldots, p_m\in (1,\infty)$, if for an $m$-(sub)linear operator $T$ and all weights $w_j^{p_j}\in A_{p_j}$ we have
\[
\|T(f_1,\ldots,f_m)\|_{L^{p}(w^p)}\leq C\, \prod_{j=1}^m\|f_j\|_{L^{p_j}(w_j^{p_j})},
\]
where $w= \prod_{j=1}^m w_j$ and $\frac{1}{p}=\sum_{j=1}^m\frac{1}{p_j}$, then the same estimate holds for all $p_j\in (1,\infty)$, weights $w_j^{p_j}\in A_{p_j}$ and $w$ and $p$ as before.

Considering a different kind of generalization, it was shown by Auscher and Martell \cite{AM07} that a \textit{limited range} version of the extrapolation result holds: if there are exponents $0<p_-<p_+\leq\infty$ such that the estimate \eqref{eq:extrapinput} is valid for a fixed $p_0\in (p_-,p_+)$ and all weights $w$ in the Muckenhoupt and \emph{Reverse H\"older} class $A_{p_0/p_-}\cap\RH_{(p_+/p_0)'}$, then \eqref{eq:extrapout} is valid for all $p\in (p_-,p_+)$ and all weights $w \in A_{p/p_-}\cap\RH_{(p_+/p)'}$.

Vector-valued extensions of the extrapolation theory have also been considered. Through an argument using Fubini's Theorem, the initial estimate \eqref{eq:extrapinput} immediately implies not only the estimate \eqref{eq:extrapout} for all $p \in (1,\infty)$, but also for extensions of the operator $T$ to functions taking values in the sequence spaces $\ell^s$ or more generally Lebesgue spaces $L^s$ for $s \in (1,\infty)$. Moreover, Rubio de Francia showed in \cite[Theorem 5]{Ru85} that one can take this even further. Indeed, this result states that assuming \eqref{eq:extrapinput} holds for some $p_0\in (1,\infty)$ and for all weights $w\in A_{p_0}$, then for each Banach function space $X$ with the $\UMD$ property, $T$ extends to an operator $\widetilde{T}$ on the Bochner space $L^p(X)$ which satisfies
\[
\nrmb{\widetilde{T}f}_{L^p(X)}\leq C \,\|f\|_{L^p(X)}
\]
for all $p\in(1,\infty)$ and all $f \in L^p(X)$ . Recently, it was shown by Amenta, Veraar, and the first author in \cite{ALV17} that given $p_-\in(0,\infty)$, if \eqref{eq:extrapinput} holds for $p_0\in (p_-,\infty)$ and all weights $w\in A_{p_0/p_-}$, then for each Banach function space $X$ such that $X^{p_-}$ has the $\UMD$ property, $T$ extends to an operator $\widetilde{T}$ on the Bochner space $L^p(w;X)$ and satisfies
\[
\nrmb{\widetilde{T}f}_{L^p(w;X)}\leq C \, \|f\|_{L^p(w;X)}
\]
for all $p\in(p_-,\infty)$, all weights $w \in A_{p/p_-}$ and all $f \in L^p(w;X)$. Here $X^{p_-}$ is the $p_-$-concavification of $X$, see Section \ref{sec:preliminaries} for the definition.

 Vector-valued estimates in harmonic analysis have been actively developed in the past decades. Important for the mentioned vector-valued extrapolation are the equivalence of the boundedness of the vector-valued Hilbert transform on $L^p(X)$  and the $\UMD$ property of $X$ for a Banach space $X$ (see \cite{Bo83,Bu83}) and the fact that for a Banach function space $X$ the $\UMD$ property implies the boundedness of the lattice Hardy--Littlewood maximal operator on $L^p(X)$ (see \cite{Bo84,Ru86}). For recent results in vector-valued harmonic analysis in $\UMD$ Banach function spaces, see for example \cite{BFR12,DK17, Ho14, HV15,Xu15}.

\bigskip

In the recent work \cite{CM17} by Cruz-Uribe and Martell both the limited range and the multilinear extrapolation result were combined, yielding a unified multilinear limited range version of the extrapolation result in the scalar-valued case. This result also covers vector-valued extensions to  $\ell^s$ for certain $s \in (0,\infty)$. This opened the question whether a unified multilinear limited range extrapolation theorem also holds for more general Banach function spaces. In this work, we give a positive answer to this question.

We now state our main result, in which we denote $X \in \UMD_{p_-,p_+}$ for the technical assumption that $\hab{(X^{p_-})^*}^{(p_+/p_-)'}$ has the $\UMD$ property, see Section \ref{sec:UMDp-p+} for a thorough discussion of this assumption. A more general version of this theorem can be found in Theorem \ref{thm:multi-limited-range-extrap} below.

\begin{theorem}\label{thm:maincor} Let $m\in\N$ and fix $0< p_j^-<p_j^+\leq\infty$ for $j\in\{1,\ldots,m\}$. Let $T$ be an operator defined on $m$-tuples of functions and suppose there exist $p_j\in (p_j^-,p_j^+)$ such that for all weights $w_j^{p_j} \in A_{{p_j/p_j^-}}\cap RH_{(p_j^+/p_j)'}$ and $f_j\in L^{p_j}(w_j^{p_j})$ we have
\begin{equation*}
  \nrmb{T(f_1,\ldots,f_m)}_{L^p(w^p)} \leq C\,\prod_{j=1}^m\|f_j\|_{L^{p_j}(w_j^{p_j})},
\end{equation*}
with $w=\prod_{j=1}^mw_j$, $\frac{1}{p}=\sum_{j=1}^m\frac{1}{p_j}$, and where $C>0$ depends only on the characteristic constants of the weights. Moreover, assume that $T$ satisfies one of the following conditions:
\begin{enumerate}[(i)]
\item \label{main:1} $T$ is $m$-linear.
\item \label{main:2} $T$ is $m$-sublinear and positive valued.
\end{enumerate}
Let  $X_1,\ldots, X_m$ be quasi-Banach function spaces over a $\sigma$-finite measure space $(S,\mu)$ and define $X=X_1\cdots X_m$. Assume that for all simple functions $f_j:\R^d \to X_j$ the function $\widetilde{T}f:\R^d \to X$ given by
\begin{equation*}
  \widetilde{T}(f_1,\ldots,f_m)(x,s) := T(f_1(\cdot, s),\ldots,f_m(\cdot,s))(x), \qquad x\in \R^d, \quad s\in S
\end{equation*}
is well-defined and strongly measurable. If $X_j\in \UMD_{p_j^-,p_j^+}$, then for all $p_j\in (p_j^-,p_j^+)$ and weights $w_j^{p_j} \in A_{{p_j/p_j^-}}\cap RH_{(p_j^+/p_j)'}$, $\widetilde{T}$ extends to a bounded operator on $L^{p_1}(w_1^{p_1};X_1)\times\cdots\times L^{p_m}(w_m^{p_m};X_m)$ with
\begin{equation*}
  \nrmb{\widetilde{T}(f_1,\ldots,f_m)}_{L^p(w^p;X)} \leq C'\,\prod_{j=1}^m\|f_j\|_{L^{p_j}(w_j^{p_j};X_j)},
\end{equation*}
for all $f_j\in L^{p_j}(w_j^{p_j};X_j)$, with $w$ and $p$ are as before, and where $C'>0$ depends on the $p_j$, $p_j^-$, $p_j^+$, the characteristic constants of the weights, and the spaces $X_j$.
\end{theorem}

\begin{remark}\label{rem:mainthm}~
\begin{itemize}
\item If $T$ is a linear operator as in Theorem \ref{thm:maincor}, we have for $f_j \otimes \xi_j \in L^{p_j}(w_j^{p_j}) \otimes X_j$ that
\begin{equation*}
  \widetilde{T}(f_1\otimes \xi_1, \cdots, f_m \otimes \xi_m) = T(f_1,\cdots,f_m) \otimes \xi_1\cdots\xi_m \in L^{p}(w^{p}) \otimes X.
\end{equation*}
So in this case $\widetilde{T}$ is automatically well-defined and strongly measurable for all simple functions $f_j:\R^d \to X$.
  \item Although we state Theorem \ref{thm:maincor} for Banach function spaces, it extends to spaces isomorphic to a closed subspace of a Banach function space and by standard representation techniques also to certain Banach lattices, see \cite{LT79, MN91} for the details.
  \item In \cite{CM17} scalar-valued multilinear limited range extrapolation is proven through off-diagonal extrapolation. Relying on this result, in this paper we prove the vector-valued multilinear limited range result. Our method does not directly generalize to the off-diagonal setting, which leaves vector-valued off-diagonal extrapolation as an open problem.
  \item In Theorem \ref{thm:maincor} one could allow for $p_j^- = 0$. In this case one would have to interpret $X_j \in \UMD_{0,p_j^+}$ as $X_j \in \UMD_{p,p_j^+}$ for some $p\in (0,p_j^+)$.
\end{itemize}
\end{remark}

Even in the linear case $m=1$ our result is new in the sense that it extends the main result of \cite{ALV17} to allow for finite $p_j^+$, which yields many new applications. We are now able to consider, for example, Riesz transforms associated to elliptic operators through the weighted estimates obtained in \cite{AM07}. Many more examples of such operators can also be considered through recent advances in the theory of sparse dominations. Indeed, for example for certain Fourier multipliers such as Bochner-Riesz multipliers as well as for spherical maximal operators, sparse bounds have been found. Sparse bounds naturally imply weighted norm estimates which, through our result, yield bounded vector-valued extensions for such operators. For a more elaborate discussion as well as for references we refer the reader to Section \ref{sec:applications}.

Our result is also new in the full range multilinear case, i.e., if $p_j^-=1$, $p_j^+=\infty$ for all $j\in\{1,\ldots, m\}$. This can, for example, be applied to multilinear Calder\'on-Zygmund operators, as these satisfy the appropriate weighted bounds to apply our result. We elaborate on this in Section \ref{sec:applications}.

Finally, for the case $m=2$ our result yields new results for boundedness of the vector-valued bilinear Hilbert transform $\widetilde{\BHT}$, due to known scalar-valued weighted bounds as were first established by Culiuc, di Plinio, and Ou \cite{CDPO16}.
Bounds for the vector-valued bilinear Hilbert transform $\widetilde{\BHT}$ have useful applications in PDEs, see \cite{BM16} and references therein. The precise result we obtain can be found in Theorem \ref{thm:bilhilbo}.

\begin{remark} \label{rem:intro}
In the recent work \cite{LMO18} of Li, Martell and Ombrosi, and the recent work \cite{N18} of the second author, scalar-valued extrapolation results were obtained using the multilinear weight classes from \cite{LOPTT09}, which were made public after this paper first appeared. Rather than considering a condition for each weight individually, these weight classes allow for an interaction between the various weights, making them more appropriately adapted to the multilinear setting. This gives rise to the problem of extending these results to the vector-valued case. To facilitate this, it seems that an appropriate multilinear $\UMD$ condition on tuples of Banach function spaces is required. We leave this as a basis for future research.
\end{remark}

\bigskip

This article is organized as follows:
\begin{itemize}
  \item In Section \ref{sec:preliminaries} we summarize the preliminaries on Muckenhoupt weights, product quasi-Banach function spaces and the $\UMD$ property.
  \item In Section \ref{sec:UMDp-p+} we discus the $\UMD_{p_-,p_+}$ property and give examples of quasi-Banach function spaces satisfying the $\UMD_{p_-,p_+}$ property.
  \item In Section \ref{sec:main} we prove our main result in terms of $(m+1)$-tuples of functions, proving Theorem \ref{thm:maincor} as a corollary.
  \item In Section \ref{sec:applications} we prove new vector-valued bounds for various operators.
\end{itemize}

\textbf{Acknowledgement.} The authors thank Mark Veraar and Dorothee Frey for their helpful comments on the draft. Moreover the authors thank Alex Amenta for his suggestion to consider the multilinear setting. Finally the authors would like to thank the anonymous referees for their suggestions.

\section{Preliminaries}\label{sec:preliminaries}
\subsection{Muckenhoupt weights}
A locally integrable function $w:\R^d\to (0,\infty)$  is called a \textit{weight}. For $p \in [1,\infty)$ and a weight $w$ the space $L^p(w)$ is the subspace of all measurable functions $f: \R^d \to \C$, which we denote by $f \in L^0(\R^d)$, such that
\begin{equation*}
  \nrm{f}_{L^p(w)}:= \has{\int_{\R^d}\abs{f(x)}^pw(x)\dd x}^{1/p}<\infty.
  \end{equation*}
 By a cube $Q\subseteq\R^d$ we will mean a half-open cube whose sides are parallel to the coordinate axes and for a locally integrable function $f\in L^0(\R^d)$ we will write $\ip{f}_Q:= \frac{1}{\abs{Q}}\int_Q f \dd x$.

For $p\in[1,\infty)$ we will say that a weight $w$ lies in the \emph{Muckenhoupt class $A_p$} and write $w\in A_p$ if it satisfies
\[
[w]_{A_p}:=\sup_{Q}\ip{w}_{Q}\langle w^{1-p'}\rangle_{Q}^{p-1}<\infty,
\]
where the supremum is taken over all cubes $Q\subseteq\R^d$ and the second factor is replaced by $(\essinf_Q w)^{-1}$ if $p=1$. We define $A_\infty:=\bigcup_{p\in[1,\infty)}A_p$.
\begin{lemma}\label{lem:muckenhoupt}Let $p \in [1,\infty)$ and $w \in A_p$.
\begin{enumerate}[(i)]
\item \label{it:mw2} $w \in A_q$ for all $q \in [p,\infty)$ with $[w]_{A_q} \leq [w]_{A_p}$.
    \item \label{it:mw1} If $p>1$, $w^{1-p'} \in A_{p'}$ with $[w]^{\frac{1}{p}}_{A_p} = [w^{1-p'}]_{A_{p'}}^{\frac{1}{p'}}$.
  \item \label{it:mw3} If $p>1$, there exists an $\varepsilon>0$ such that $w \in A_{p-\varepsilon}$ and $[w]_{A_{p-\varepsilon}}\leq C_{p} \, [w]_{A_p}$.
\end{enumerate}
\end{lemma}
The first two properties of Lemma \ref{lem:muckenhoupt} are immediate from the definition. For the third see \cite[Corollary 7.2.6]{Gr14a}. The linear estimate of $[w]_{A_{p-\varepsilon}}$ in terms of $[w]_{A_p}$ can be found in \cite[Theorem 1.2]{HPR12}. Note that self-improvement properties for $A_p$ weights are classical. We opt to use this quantitative version of the result for clarity in the proof of our main theorem.

For $s\in[1,\infty)$ we say that $w \in A_\infty$ satisfies a \textit{reverse H\"{o}lder} property and write $w\in\RH_s$ if
\[
[w]_{\RH_s}:=\sup_{Q}\langle w^s\rangle^{\frac1s}_{Q}\langle w\rangle_{Q}^{-1}<\infty.
\]
We will require the following properties of the reverse H\"{o}lder classes, see \cite{JN91}.
\begin{lemma}\label{lem:reverseholder}
  Let $r \in (1,\infty)$, $s \in [1,\infty)$ and define $p = s(r-1)+1$. For $w \in A_\infty$ the following are equivalent
  \begin{enumerate}[(i)]
    \item \label{it:weightsequiv1} $w \in A_r \cap \RH_s$.
    \item \label{it:weightsequiv2} $w^s \in A_p$.
    \item \label{it:weightsequiv3} $w^{1-r'} \in A_{p'}$.
  \end{enumerate}
Moreover we have
 \begin{align*}
  \max\cbraceb{[w]_{\RH_s}^s, [w]_{A_r}^s} &\leq [w^s]_{A_p} \leq \hab{[w]_{A_r}[w]_{\RH_s}}^s.
  \end{align*}
\end{lemma}

For $n\in\N$ we will write $\phi_{a,b,\cdots}$ for a non-decreasing function $[1,\infty)^n \to [1,\infty)$, depending on the parameters $a,b,\cdots$ and the dimension $d$. This function may change from line to line. We need non-decreasing dependence on the Muckenhoupt characteristics in our proofs. In \cite[Appendix A]{ALV17} it is shown how to deduce non-decreasing dependence from a more general estimate in terms of the Muckenhoupt characteristics.

\subsection{Banach function spaces}
Let $(S,\mu)$ be a $\sigma$-finite measure space. A subspace $X$ of $L^0(S)$ equipped with a quasi-norm $\nrm{\, \cdot \,}_X$ is called a \emph{quasi-Banach function space} if it satisfies the following properties:
\begin{enumerate}[(i)]
\item If $\xi\in L^0(S)$ and $\eta\in X$ with $\abs{\xi} \leq \abs{\eta}$, then $\xi\in X$ and $\nrm{\xi}_X \leq \nrm{\eta}_X$.
\item There is an $\xi\in X$ with $\xi>0$.
\item If $0
\leq \xi_n \uparrow \xi$ with $(\xi_n)_{n=1}^\infty$ a sequence in $X$, $\xi \in L^0(S)$ and $\sup_{n \in \N}\nrm{\xi_n}_X < \infty$, then $\xi\in X$ and $\nrm{\xi}_X = \sup_{n\in \N}\nrm{\xi_n}_X$.
\end{enumerate}
It is called a \emph{Banach function space} if $\nrm{\, \cdot \,}_X$ is a norm. A Banach function space $X$ is called \textit{order continuous} if for any sequence $0 \leq \xi_n \uparrow \xi \in X$ it holds that $\nrm{\xi_n -\xi}_X \to 0$. Order continuity of a Banach function space $X$ ensures that its dual $X^*$ is again a Banach function space (see \cite[Section 1.b]{LT79}), and that the Bochner space $L^p(S';X)$ is a Banach function space over $(S \times S',\mu\times\mu')$ for any $\sigma$-finite measure space $(S',\mu')$. As an example we note that any reflexive Banach function space is order-continuous.

A quasi-Banach function space $X$ is said to be \emph{$p$-convex} for $p \in (0,\infty]$ if
\begin{equation*}
  \nrms{\has{\sumkn \abs{\xi_k}^p}^\frac{1}{p}}_X \leq \has{\sumkn\nrm{\xi_k}_X^p}^\frac{1}{p}
\end{equation*}
for all $\xi_1,\cdots,\xi_n \in X$ with the usual modification when $p = \infty$. It is said to be \emph{$p$-concave} when the reverse inequality holds. Usually the defining inequalities for $p$-convexity and $p$-concavity include a constant depending on $p$ and $X$, but as shown in \cite[Theorem 1.d.8]{LT79}, $X$ can be renormed equivalently such that these constants equal 1. See \cite[Section 1.d]{LT79} for a thorough introduction of $p$-convexity and concavity in Banach function spaces and see \cite{Ka84} for the quasi-Banach function space case.

We define the \emph{$p$-concavification} of a quasi-Banach function space $X$  for $p \in (0,\infty)$ by
\begin{equation*}
  X^p := \cbrace{\xi \in L^0(S): \abs{\xi}^\frac{1}{p} \in X} = \cbrace{\abs{\xi}^p \sgn{\xi}: \xi \in X},
\end{equation*}
equipped with the quasi-norm $\nrm{\xi}_{X^p} := \nrmb{\abs{\xi}^{\frac{1}{p}}}^p_X$. Note that $X^p$ is a Banach function space if and only if $X$ is $p$-convex. In particular, $X$ is a Banach function space if and only if it is $1$-convex.

For two quasi-Banach function spaces $X_0,X_1$ over the same measure space $(S,\mu)$ we define the vector space $X_0 \cdot X_1$ as
\begin{equation*}
  X_0 \cdot X_1 := \cbrace{\xi_0\cdot \xi_1:\xi_0 \in X_0, \xi_1 \in X_1}
\end{equation*}
and for $\xi\in X_0 \cdot X_1$ we define
\begin{equation*}
  \nrm{\xi}_{X_0 \cdot X_1}:= \inf \cbraceb{\nrm{\xi_0}_{X_0}\nrm{\xi_1}_{X_1}:\abs{\xi} = \xi_0\cdot \xi_1, 0 \leq \xi_0 \in X_0, 0\leq \xi_1 \in X_1}
\end{equation*}
We call $X_0 \cdot X_1$ a \emph{product quasi-Banach function space} if $\nrm{\, \cdot \,}_{X_0 \cdot X_1}$ defines a complete quasi-norm on $X_0 \cdot X_1$.  We will mostly be working with so called \emph{Calder\'on-Lozanovskii} products. These are product quasi-Banach function spaces of the form $X_0^{1-\theta}\cdot X_1^\theta$ for some $\theta \in (0,1)$, see \cite{Ca64,Lo69}. Of course the definition of product quasi-Banach function spaces and Calder\'on-Lozanovskii products can be canonically extended to $m$ quasi-Banach function spaces over the same measure space for any $m \in \N$. We give a few examples of product Banach function spaces, see also \cite{Bu87}.
\begin{example}\label{ex:exproducts}
  Fix $m \in \N$ and let $(S,\mu)$ be an atomless or atomic $\sigma$-finite measure space.
\begin{enumerate}[(i)]
  \item \label{it:exproduct1} Lebesgue spaces: $L^p(S) = L^{p_1}(S) \cdots L^{p_m}(S)$ for $p_j \in (0,\infty)$ and $\frac{1}{p} = \sum_{j=1}^m\frac{1}{p_j}$.
  \item \label{it:exproduct2} Lorentz spaces: $L^{p,q}(S) = L^{p_1,q_1}(S) \cdots L^{p_m,q_m}(S)$ for $p_j,q_j \in (0,\infty)$, $\frac{1}{p} = \sum_{j=1}^m\frac{1}{p_j}$ and $\frac{1}{q} = \sum_{j=1}^m\frac{1}{q_j}$.
  \item \label{it:exproduct3} Orlicz spaces: $L^{\Phi}(S) = L^{\Phi_1}(S)\cdots L^{\Phi_m}(S)$ for Young functions $\Phi_j$ and $\Phi^{-1} = \Phi_1^{-1} \cdots \Phi_m^{-1}$.
\end{enumerate}
\end{example}

We will use the following properties of product Banach function spaces:
\begin{proposition}\label{prop:pBFS} Let $X,X_0,X_1$ be Banach function spaces over a $\sigma$-finite measure space $(S,\mu)$ and let $\theta \in (0,1)$.
\begin{enumerate}[(i)]
    \item \label{it:pBFScomplex} If $X_0$ or $X_1$ is reflexive, then $X_0^{1-\theta}\cdot X_1^\theta=[X_0,X_1]_\theta$.
    \item \label{it:pBFSrefl} If $X_0$ or $X_1$ is reflexive, then $X_0^{1-\theta}\cdot X_1^\theta$ is reflexive.
    \item \label{it:pBFSdual}  $\hab{X_0^{1-\theta}\cdot X_1^\theta}^* = \hab{X_0^*}^{1-\theta}\cdot \hab{X_1^*}^\theta$.
    \item \label{it:pBFSdualconcave} $(X^{\theta})^* = (X^*)^{\theta} \cdot L^{1/(1-\theta)}(S)$.
    \item \label{it:pBFSUMD} If $X_0$ and $X_1$ have the $\UMD$ property, then $X_0^{1-\theta}\cdot X_1^\theta$ has the $\UMD$ property.
  \end{enumerate}
\end{proposition}
Part \ref{it:pBFScomplex} follows from \cite{Ca64}, it has been extended to the product quasi-Banach function space setting in \cite{KMM07,KM98}. Part \ref{it:pBFSrefl} is proven in \cite[Theorem 3]{Lo69}. It also follows from \cite{Ca64} through complex interpolation. Part \ref{it:pBFSdual} is proven in \cite[Theorem 2]{Lo69} and for \ref{it:pBFSdualconcave} see \cite[Theorem 2.9]{Sc10}. Finally part \ref{it:pBFSUMD} follows from part \ref{it:pBFScomplex} and \cite[Proposition 4.2.17]{HNVW16}, see also the next section on the $\UMD$ property.

\subsection{The \texorpdfstring{$\UMD$}{UMD} property}
We say that a Banach space $X$ has the $\UMD$ property if the martingale difference sequence of any finite martingale in $L^p(\Omega;X)$ is unconditional for some (equivalently all) $p \in (1,\infty)$. The $\UMD$ property implies reflexivity and if $X$ has the $\UMD$ property, then $X^*$ has the $\UMD$ property as well. Standard examples of Banach spaces with the $\UMD$ property include reflexive $L^p$-spaces, Lorentz spaces, Orlicz spaces,  Sobolev spaces, Besov spaces and Schatten classes. For a thorough introduction to the theory of $\UMD$ spaces we refer the reader to \cite{Bu01,HNVW16}.

Throughout this paper we will consider Banach function spaces with the $\UMD$ property.
  In this case we have a characterisation of the $\UMD$ property in terms of the lattice Hardy-Littlewood maximal operator, which for simple $f:\R^d \to X$ is defined by
 \begin{equation*}
   \widetilde{M}f(x) := \sup_{Q} \ip{\abs{f}}_{Q} \ind_{Q}(x),
 \end{equation*}
 where the supremum is taken over all cubes $Q \subseteq\R^d$ (see \cite{GMT93} for the details). The boundedness of $\widetilde{M}$ on both $L^p(\R^d;X)$ and $L^p(\R^d;X^*)$ for some (equivalently all) $p \in (1,\infty)$ is equivalent to $X$ having the $\UMD$ property by a result of Bourgain \cite{Bo84} and Rubio de Francia \cite{Ru86}. Moreover, if $X$ has the $\UMD$ property we have the following weighted bound for all $p \in (1,\infty)$, $w \in A_p$ and $f \in L^p(w;X)$
 \begin{equation}\label{eq:maximalbdd}
   \nrmb{\widetilde{M}f}_{L^p(w;X)} \leq \inc_{X,p}([w]_{A_p}) \nrmb{f}_{L^p(w;X)},
 \end{equation}
 see \cite{GMT93}. A more precise dependence on the weight characteristic can be found in \cite{HL17}.

The $\UMD$ property of a Banach function space $X$ implies that certain $q$-concavifications of $X$ also have the $\UMD$ property, see \cite[Theorem 4]{Ru86}.
\begin{proposition}[Rubio de Francia]\label{prop:UMDopenRubio}
Let $X$ be a Banach function space over a $\sigma$-finite measure space $(S,\mu)$ such that $X$ has the $\UMD$ property. Then there exists an $\varepsilon>0$ such that such that $X^q$ has the $\UMD$ property for all $0<q<1+\varepsilon$.
\end{proposition}
Note that the difficult part of Proposition \ref{prop:UMDopenRubio} is the claim that $X^q$ has the $\UMD$ property for  $1<q<1+\varepsilon$.

\section{The \texorpdfstring{$\UMD_{p_-,p_+}$}{UMDp+p-} property of quasi-Banach function spaces}\label{sec:UMDp-p+}
For our main result we need an extension of the $\UMD$ property, as we will often consider quasi-Banach function spaces of which a concavification has the $\UMD$ property. In particular, we will use the following notion:
\begin{definition}\label{def:UMDp-p+}
   Let $X$ be a quasi-Banach function space and let $0<p_-<p_+\leq\infty$. Then we say $X$ has the $\UMD_{p_-,p_+}$ property if and only if $X$ is $p_-$-convex, $p_+$-concave and $\hab{(X^{p_-})^*}^{(p_+/p_-)'}$ has the $\UMD$ property. We denote this by $X \in \UMD_{p_-,p_+}$.
\end{definition}
Note that $X$ is a Banach function space with the $\UMD$ property if and only if $X \in \UMD_{1,\infty}$ and we denote this by $X \in \UMD$.

\begin{remark}~\label{rem:UMDp-p+}
\begin{itemize}
  \item The $p_-$-convexity in Definition \ref{def:UMDp-p+} implies that $X^{p_-}$ is a Banach function space, so its dual $(X^{p_-})^*$ is non-trivial. Moreover $(X^{p_-})^*$ is a Banach function space, since it has the $\UMD$ property by Proposition \ref{prop:UMDopenRubio} and is therefore reflexive, which implies that $X^{p_-}$ is order-continuous.
  \item The $p_+$-concavity assumption in Definition \ref{def:UMDp-p+} is not restrictive, as any quasi-Banach function space with the $\UMD$ property is actually isomorphic to a Banach function space (see \cite{CCV18}), which implies that $(X^{p_-})^*$ is $(p_+/p_-)'$-convex and thus that $X$ is $p_+$-concave by \cite[Section 1.d]{LT79}
\end{itemize}
\end{remark}

We first show some basic results for the $\UMD_{p_-,p_+}$ property.
\begin{proposition} Fix $0<p_-<p_+ \leq \infty$ and let $X$ be a quasi-Banach function space over a $\sigma$-finite measure space $(S,\mu)$ such that $X \in \UMD_{p_-,p_+}$.
\begin{enumerate}[(i)]\label{prop:UMDp+p-}
\item \label{it:UMDrescale}$X^{p_-} \in \UMD_{1,p_+/p_-}$.
\item \label{it:UMDdual}$X^* \in  \UMD_{p_+',p_-'}$ if $p_- \geq 1$.
\item \label{it:UMDinterval}$X \in \UMD_{\tilde{p}_-,\tilde{p}_+}$ for all $\tilde{p}_- \in (0,p_-]$ and $\tilde{p}_+ \in [p_+,\infty]$.
\item \label{it:UMDinterp} If $1<p_-<p_+<\infty$, then $X = [Y,L^2(S)]_{\theta}$ for a Banach function space $Y \in \UMD$ and $\theta = 2/\max\cbrace{p_-',p_+}$.
\item \label{it:UMDLp}$L^p(S';X) \in \UMD_{p_-,p_+}$ for all $p \in (p_-,p_+)$ and any $\sigma$-finite measure space $(S',\mu')$.
\end{enumerate}
\end{proposition}

\begin{proof}
  Part \ref{it:UMDrescale} follows directly from the definition. For part \ref{it:UMDdual} the $p_+'$-convexity and $p_-'$-concavity follow from \cite[Section 1.d]{LT79}. If $p_- =1$, the claim is that $\hab{(X^*)^{p_+'}}^* \in \UMD$, which is clear. Assuming $p_->1$, we have by Proposition \ref{prop:pBFS}
\begin{align*}
  \hab{\hab{(X^*)^{p_+'}}^*}^{(p'_-/p_+')'} &= \has{\has{\hab{(X^{p_-})^*}^{p_+'/p_-}\cdot L^{p_-'/p_+'}(S)}^*}^{(p'_-/p_+')'} \\
  &= \has{\has{\hab{(X^{p_-})^*}^{(p_+/p_-)'\theta}\cdot L^1(S)^{1-\theta}}^*}^{1/\theta}\\
  &= \has{\hab{(X^{p_-})^*}^{(p_+/p_-)'}}^*\cdot  \hab{L^\infty(S)}^{(1-\theta)/\theta}\\
  &= \has{\hab{(X^{p_-})^*}^{(p_+/p_-)'}}^*
\end{align*}
with $\theta := \frac{1}{(p'_-/p_+')'}<1$, since taking a product with $L^\infty(S)^{(1-\theta)/\theta}=L^\infty(S)$ has no effect on the space. Thus we conclude that $X^* \in  \UMD_{p_+',p_-'}$.

  For part \ref{it:UMDinterval} the $p_+'$-convexity and $p_-'$-concavity follow from \cite[Theorem 4.2]{Ma04}. First assume that $p_-=1$ and let $\tilde{p}_- \in (0,1)$.
  By Proposition \ref{prop:pBFS}\ref{it:pBFSdualconcave} we have
  \begin{align}\label{eq:UMDp-p+interval}
    \hab{(X^{\tilde{p}_-})^*}^{(p_+/\tilde{p}_-)'} = \ha{X^*}^{\tilde{p}_-(p_+/\tilde{p}_-)' } \cdot L^{\frac{p_+-\tilde{p}_-}{(1-\tilde{p}_-)p_+}}(S) = (X^*)^{p_+'\theta} \cdot L^{1}(S)^{1-\theta}
  \end{align}
  with
  \begin{equation*}
    \theta := \frac{\tilde{p}_-(p_+-1)}{p_+-\tilde{p}_-} <1.
  \end{equation*}
  By assumption $(X^*)^{p_+'} \in \UMD$, so
  \begin{equation*}
    \bracb{(X^*)^{p_+},L^1(S)}^*_\theta = \bracb{\hab{(X^*)^{p_+'}}^*,L^\infty(S)}_\theta = \has{\hab{(X^*)^{p_+'}}^*}^\theta \in \UMD
  \end{equation*}
  by Proposition \ref{prop:pBFS} and Proposition \ref{prop:UMDopenRubio}. Using Proposition \ref{prop:pBFS}\ref{it:pBFScomplex}, we obtain from \eqref{eq:UMDp-p+interval} that $X \in \UMD_{\tilde{p}_-,p_+}$.
  For arbitrary $0<p_-<p_+ \leq \infty$ we know that $X \in \UMD_{\tilde{p}_-,p_+}$ for all $\tilde{p}_- \in (0,p_-]$ by \ref{it:UMDrescale} and Proposition \ref{prop:UMDopenRubio} yields that $X \in \UMD_{\tilde{p}_-,\tilde{p}_+}$ for all $\tilde{p}_+ \in [p_+,\infty]$.

For part \ref{it:UMDinterp} note that $X \in \UMD_{p',p}$ with $p = \max\cbrace{p_-',p_+}$ by part \ref{it:UMDinterval}. Therefore
\begin{equation*}
Y  := \has{\hab{(X^{p'})^*}^{(p/p')'}}^* \in \UMD.
\end{equation*}
Then using Proposition \ref{prop:pBFS} we have
\begin{equation*}
  X = \has{\hab{\ha{X^{p'}}^*}^{1/{p'}}\cdot L^{p}(S)}^* = \bracb{ \hab{\ha{X^{p'}}^*}^{(p/p')'}, L^2(S)}_{2/p}^* = \brac{ Y , H}_{\theta}.
\end{equation*}

  Finally part \ref{it:UMDLp} follows from \cite[Proposition 4.2.15]{HNVW16} as
  \begin{equation*}
    \hab{\hab{L^p(S';X)^{p_-}}^*}^{(p_+/p_-)'} = L^{\frac{(p_+-p_-)p}{p_+(p-p_-)}}\hab{S';\hab{(X^{p_-})^*}^{(p_+/p_-)'}}. \qedhere
  \end{equation*}
\end{proof}

Next we note how product quasi-Banach function spaces work under the $\UMD_{p_-,p_+}$ property. In particular the following result describes some properties of the space $X$ in our main theorem, Theorem \ref{thm:maincor}.
\begin{proposition}
  Let $X_1,\cdots,X_m$ be quasi-Banach function spaces. For $j=1,\cdots,m$ let $0<p_j^-<p_j^+\leq \infty$ and assume that $X_j \in \UMD_{p_j^-,p_j^+}$. Let $X = X_1\cdots X_m$, then $X \in \UMD_{p_-,p_+}$, where $\frac{1}{p_-}:= \sum_{j=1}^m \frac{1}{p_j^-}$ and $\frac{1}{p_+}:= \sum_{j=1}^m \frac{1}{p_j^+}$.
\end{proposition}

\begin{proof}
  We will prove the proposition for $m=2$. The general case can be proven by induction, cf. the proof of Lemma \ref{lemma:pBFSbochner}. First note that $X^{p_-} = X_1^{p_1^-(p_-/p_1^-)}\cdot X_2^{p_2^-(p_-/p_2^-)}$ is a Banach function space by assumption, so $X$ is $p_-$-convex. By Proposition \ref{prop:pBFS} we have
  \begin{align*}
    \hab{(X^{p_-})^*}^{(p_+/p_-)'} &= \hab{(X_1^{p_-} \cdot X_2^{p_-})^*}^{(p_+/p_-)'}\\ &= \hab{(X_1^{p^-_1})^*}^{(p_1^+/p_1^-)'(1-\theta)} \cdot \hab{(X_2^{p^-_2})^*}^{(p_2^+/p_2^-)'\theta}\\
  \end{align*}
  with
  \begin{equation*}
    \theta = \frac{\frac{1}{p_2^-}-\frac{1}{p_2^+}}{\frac{1}{p_-}-\frac{1}{p_+}}.
  \end{equation*}
  Thus by Proposition \ref{prop:pBFS}\ref{it:pBFSUMD} and Remark \ref{rem:UMDp-p+} we know that $X \in \UMD_{p_-,p_+}$.
\end{proof}

The $\UMD_{p_-,p_+}$ property of a quasi-Banach function space $X$ looks quite technical. However, as we will see in the next example, this abstract assumption is quite natural for concrete examples of Banach function spaces.

\begin{example}\label{ex:concavification}
Let $0 < p_- <p_+\leq \infty$ and let $X$ be a quasi-Banach function space over an atomless or atomic $\sigma$-finite measure space $(S,\mu)$. Then $X \in  \UMD_{p_-,p_+}$ in each of the following cases:
\begin{enumerate}[(i)]
  \item \label{it:exconcavification1} The Lebesgue spaces $X=L^p(S)$ for $p \in (p_-,p_+)$.
  \item \label{it:exconcavification2} The Lorentz spaces $X=L^{p,q}(S)$ with $p,q \in (p_-,p_+)$.
  \item \label{it:exconcavification3} The Orlicz spaces $X=L^\Phi(S)$ for which $t \mapsto \Phi(t^{1/p})$ is a convex function and $t\mapsto \Phi(t^{1/q})$ is a concave function with $p,q \in (p_-,p_+)$.
\end{enumerate}
 Note that Theorem \ref{thm:maincor} for the Lebesgue spaces described in Example \ref{ex:concavification}\ref{it:exconcavification1} follows directly from scalar-valued limited range extrapolation using Fubini's theorem, see also \cite{CMP11}.
\end{example}

\begin{proof}
Note that \ref{it:exconcavification1} is a special case of \ref{it:exconcavification2}. For \ref{it:exconcavification2} the $p_-$-convexity and $p_+$-concavity follow from \cite[Theorem 4.4 and Theorem 5.1]{Ma04}. Furthermore by the definition of $L^{p,q}(S)$ and the duality of Lorentz spaces (see \cite{Hu66}) we have that
  \begin{equation*}
    \bigl((X^{p_-})^*\bigr)^{(p_+/p_-)'} = \bigl(\hab{L^{p/{p_-},q/p_-}(S)}^*\bigr)^{(p_+/p_-)'} = L^{\frac{(p_+-p_-)p}{p_+(p-p_-)},\frac{(p_+-p_-)q}{p_+(q-p_-)}}(S).
  \end{equation*}
Since $L^{r,s}(S) \in \UMD$ for $r,s \in (1,\infty)$ (see \cite{HNVW16}), this proves \ref{it:exconcavification2}.

 For \ref{it:exconcavification3} note that $L^\Phi(S)$ is $p$-convex and $q$-concave by \cite{Ka98}. So $Y := \bigl((X^{p_-})^*\bigr)^{(p_+/p_-)'}$ is  $\frac{(p_+-p_-)p}{p_+(p-p_-)}$-convex and $\frac{(p_+-p_-)q}{p_+(q-p_-)}$-concave. By \cite[Theorem 1.f.1]{LT79} this implies that both $Y$ and $Y^*$ are uniformly convex. Note that $Y$ is an Orlicz space with Young function $\Psi(t) = \varphi^*(t^{(p_+-p_-)/p_-})$, where $\varphi(t) = \Phi(t^{1/p_-})$, and $Y^*$ is an Orlicz space with Young function $\Psi^*$. Therefore we know by \cite[Proposition 1]{Ka98} that both $\Psi$ and its conjugate function $\Psi^*$  satisfy the $\Delta_2$-condition. Thus it follows from \cite[Theorem 6.2]{FG91} that $Y \in \UMD$.
\end{proof}

We end our discussion of the $\UMD_{p_-,p_+}$ property by extending the result of Rubio de Francia for the $\UMD$ property of Banach function spaces in Proposition \ref{prop:UMDopenRubio} to the $\UMD_{p_-,p_+}$ property of quasi-Banach function spaces.

\begin{theorem}\label{thm:UMDopen}
Let $0<p_-<p_+ \leq \infty$ and let $X$ be a quasi-Banach function space over a $\sigma$-finite measure space $(S,\mu)$ such that $X \in \UMD_{p_-,p_+}$. Then there exists an $\varepsilon>0$ such that such that $X \in \UMD_{p_-q_-,p_+/q_+}$ for all $0<q_-,q_+<1+\varepsilon$.
\end{theorem}
\begin{proof}
By Proposition \ref{prop:UMDp+p-}\ref{it:UMDrescale} we may assume $p_-=1$ without loss of generality. Note that the case $p_+ = \infty$ was already included in Proposition \ref{prop:UMDopenRubio}, so we restrict our attention to $p_+< \infty$.

Applying Proposition \ref{prop:UMDopenRubio} to $(X^*)^{p'}$ yields an $r_1>1$ such that $(X^*)^{p_+'r_1} \in \UMD$. Furthermore since $p_+'>1$ we know that $X^* \in \UMD$ and thus also $X \in \UMD$. So by Proposition \ref{prop:UMDopenRubio} applied to $X$ there exists an $r_2>1$ such that $X^{r_2} \in \UMD$.  Define $r = \min\cbrace{r_1,r_2, 1+\frac{1}{p_+'}}$.

Let $\theta = \frac{r'}{p_++r'} \in (0,1)$ and define the complex interpolation space
\begin{equation*}
  Y := \brac*{(X^r)^*,(X^*)^{p_+'r}}_\theta.
\end{equation*}
Note that since  $(X^r)^*, (X^*)^{p_+'r} \in \UMD$, we know by Proposition \ref{prop:pBFS}\ref{it:pBFSUMD} that $Y \in \UMD$ as well. Moreover using Proposition \ref{prop:pBFS} we have
\begin{align*}
  Y &= \bigl((X^r)^*\bigr)^{1-\theta} \cdot \bigl((X^*)^{p_+'r}\bigr)^\theta\\
  &= \bigl((X^r)^*\bigr)^{1-\theta} \cdot \has{\bigl((X^r)^{1/r}\bigr)^*}^{p_+'r \theta}\\
  &= \hab{(X^r)^*}^{\frac{p_+}{p_++r'}} \cdot \has{\bigl((X^r)^*\bigr)^{1/r}\cdot L^{r'}(S)}^{\frac{p_+'rr'}{p_++r'}}\\
  &= \hab{(X^r)^*}^{\frac{p_++p_+'r'}{p_++r'}} \cdot L^{\frac{p_++r'}{p_+'r}}(S)
\end{align*}
Define
\[
  \alpha :=  \frac{p_++p_+'r'}{p_++r'},\qquad\beta := \frac{p_+'r}{p_++r'} <   p_+' (r-1)<1.
\]
Again by Proposition \ref{prop:pBFS} we have
\begin{align*}
  Y &= \bigl((X^r)^*\bigr)^\alpha \cdot L^{\frac{1}{\beta}}(S) =  \has{\bigl((X^r)^*\bigr)^{\frac{\alpha}{\alpha+\beta}}\cdot L^{\frac{\alpha+\beta}{\beta}}(S)}^{\alpha+\beta}  = \bigl((X^{\frac{r\alpha}{\alpha+\beta}})^*\bigr)^{\alpha+\beta}.
\end{align*}
Take $q_- = \frac{r\alpha}{\alpha+\beta}$.  Then we have
\begin{align}\label{eq:q-equality}
  q_- = \frac{r \alpha}{\alpha+\beta} = \frac{r(p_++p_+'r')}{p_++p_+'r'+p_+'r} = \frac{p_+r+p_+'r'+p_+'r}{p_++p_+'r'+p_+'r} >1.
\end{align}
Moreover
\begin{equation*}
  \alpha+\beta - \frac{r\alpha}{p_+} = \frac{p_++p_+'r'+p_+'r-r - (p_+'-1)(r+r')}{p_++r'} =1
\end{equation*}
and therefore
\begin{align*}
 (p_+/q_-)'&= \frac{\alpha+\beta}{\alpha+\beta - \frac{r\alpha}{p_+}} = \alpha+\beta.
\end{align*}
So $Y = \hab{(X^{q_-})^*}^{(p_+/q_-)'}$ and since $Y \in \UMD$, this implies that $X \in \UMD_{q_-,p_+}$. By applying Proposition \ref{prop:UMDopenRubio} once more, we can find a $q_+>1$ such that $X \in \UMD_{q_-,p_+/q_+}$. By Proposition \ref{prop:UMDp+p-}\ref{it:UMDinterval}, this completes the proof with $\varepsilon = \min\cbrace{q_--1,q_+-1}>0$.
\end{proof}

\section{Proof of the main result}\label{sec:main}
In this section we will prove our main result, Theorem \ref{thm:maincor}. The proof of Theorem \ref{thm:maincor} consists of following ingredients:
\begin{itemize}
  \item The extension of Rubio de Francia's result for the $\UMD$ property to the setting of the $\UMD_{p_-,p_+}$ property, proven in Theorem \ref{thm:UMDopen}.
  \item A vector-valued Rubio de Francia iteration algorithm, see Lemma \ref{lem:rubioalgoritme}.
  \item A result for the product of weighted Bochner spaces, proven below in Lemma \ref{lemma:pBFSbochner}.
\end{itemize}
We start with the Rubio de Francia iteration algorithm lemma. We remark that Rubio de Francia iteration algorithms also play a key role in scalar-valued extrapolation, see for example \cite{CMP11}. Recall that we write $\phi_{a,b,\cdots}$ for a non-decreasing function $[1,\infty)^{2} \to [1,\infty)$, depending on the parameters $a,b,\cdots$ and the dimension $d$.
\begin{lemma}\label{lem:rubioalgoritme}
Fix $1<r<r_+\leq \infty$ and let $Y$ be a Banach function space over a $\sigma$-finite measure space $(S,\mu)$ with $Y \in \UMD_{r_+',\infty}$. For all $w \in A_r \cap RH_{(r_+/r)'}$ and nonnegative $u \in L^{r'}(w;Y)$ there is a nonnegative $v \in L^{r'}(w;Y)$ such that $u \leq v$, $\nrm{v}_{L^{r'}(w;Y)} \leq 2\nrm{u}_{L^{r'}(w;Y)}$, and $v(\cdot,s)w \in A_1\cap\RH_{r_+'}$ with
\begin{equation*}
  \max\cbraceb{[v(\cdot,s)w]_{A_1},[v(\cdot,s)w]_{\RH_{r_+'}} } \leq \inc_{Y,r,r_+}\hab{[w]_{A_r}, [w]_{RH_{(r_+/r)'}}}
\end{equation*}
for $\mu$-a.e $s \in S$.
\end{lemma}

\begin{proof}
  Fix $w \in A_r \cap RH_{(r_+/r)'}$ and $u \in L^{r'}(w;Y)$. Define
 \begin{align*}
   u_w &:= (uw)^{r_+' } \qquad \text{and} \qquad
   X := L^{r'/r_+'}\hab{w^{1-r'};Y^{r_+'}}.
 \end{align*}
 Then $u_w \in X$. By Lemma \ref{lem:reverseholder}\ref{it:weightsequiv3} we know that for $p:=\ha*{{r_+/r}}'(r-1)+1$ we have $w^{1-r'} \in A_{p'}$ with
   \begin{equation*}
     [w^{1-r'}]_{A_{p'}}^{\frac{1}{p'}}=\bracb{w^{(r_+/r)'}}_{A_p}^{\frac{1}{p}} \leq \hab{[w]_{A_r}[w]_{\RH_{(r_+/r)'}}}^{(r_+/r)'}
   \end{equation*}
   So since
   \begin{align*}
   p' = 1+\frac{1}{\ha*{{r_+/r}}'(r-1)} =
   \frac{r'}{r_+'}
   \end{align*}
   we know that $\widetilde{M}$ is bounded on $X$ by \eqref{eq:maximalbdd} with
   \begin{equation}\label{eq:maximalbddweight}
     \bigl\|\widetilde{M}\bigr\|_{X \to X} \leq \inc_{Y,r,r_+}\bigl([w]_{A_r} ,[w]_{\RH_{(r_+/r)'}}\bigr).
   \end{equation}
   Define
   \begin{equation*}
     v := w^{-1}\cdot\has{\sum_{n=0}^\infty \frac{\widetilde{M}^nu_w}{\hab{2\bigl\|\widetilde{M}\bigr\|_{X \to X}}^n}}^{1/r_+'}
   \end{equation*}
   where $\widetilde{M}^n$ is given by $n$ iterations of $\widetilde{M}$. As $\widetilde{M}^n u_w$ is nonnegative we know that $u \leq v$. Furthermore $v \in L^{r^\prime}(w;Y)$ with
   \begin{align*}
     \nrm{v}_{L^{r'}(w;Y)} = \nrms{\sum_{n=0}^\infty \frac{\widetilde{M}^nu_w}{\hab{2\bigl\|\widetilde{M}\bigr\|_{X \to X}}^n}}_X^{1/r_+'} \leq 2\nrm{u_w}_X^{1/r_+'} = 2\nrm{u}_{L^{r'}(w;Y)}.
   \end{align*}
 Moreover, since
 \begin{equation*}
     \widetilde{M}\bigl((vw)^{r_+'}\bigr)(\cdot,s) \leq 2 \bigl\|\widetilde{M}\bigr\|_{X \to X} (v(\cdot,s)w)^{r_+'},
   \end{equation*}
 we know that $(v(\cdot,s)w)^{r_+'} \in A_1$ for $\mu$-a.e $s \in S$. Thus it follows from \eqref{eq:maximalbddweight} and Lemma \ref{lem:reverseholder} that $v(\cdot,s)w \in A_1\cap\RH_{r_+'}$  with
\begin{equation*}
 \max\cbraceb{[v(\cdot,s)w]_{A_1},[v(\cdot,s)w]_{\RH_{r_+'}} } \leq \inc_{Y,r,r_+}\bigl([w]_{A_r} ,[w]_{\RH_{(r_+/r)'}}\bigr)
\end{equation*}
for $\mu$-a.e $s \in S$
\end{proof}

Next we prove the  result for the product of weighted Bochner spaces, which follows from the properties of product Banach function spaces in Proposition \ref{prop:pBFS} and complex interpolation of weighted Bochner spaces.
\begin{lemma}\label{lemma:pBFSbochner}
   Fix $m\in \N$ and $r \in (1,\infty)$. Let $Y_1,\cdots,Y_m$ be reflexive Banach function spaces over a $\sigma$-finite measure space $(S,\mu)$, let $w_1,\cdots,w_m$ be weights and take $\theta_1,\cdots,\theta_m \in (0,1)$ such that $\sum_{j=1}^m \theta_j = 1$. Define $Y = Y_1^{\theta_1}\cdots Y_m^{\theta_m}$ and $w = \prod_{j=1}^m w_j^{\theta_j}$. Then we have
  \begin{equation*}
    L^r(w;Y)=L^{r}\hab{w_1;Y_1}^{\theta_1} \cdots L^{r}\hab{w_m;Y_m}^{\theta_m}
  \end{equation*}
\end{lemma}

\begin{proof}
  We will prove the lemma by induction. For $m=1$ the result is trivial. Now assume that the statement holds for $m=k-1$ for some $k \in \N$. We will show the statement for $m=k$.

  Let $\tilde{\theta}_j = \frac{\theta_j}{1-\theta_k}$ for $j=1,\cdots,k-1$ and define
  \begin{equation*}
    X = Y_1^{\tilde{\theta}_1}\cdots Y_{k-1}^{\tilde{\theta}_{k-1}},\qquad v = \prod_{j=1}^m w_j^{\tilde{\theta}_j}.
  \end{equation*}
 Using Proposition \ref{prop:pBFS}\ref{it:pBFScomplex} twice and complex interpolation of weighted Bochner spaces (see
\cite[Theorem 1.18.5]{Tr78} and \cite{Bu87}) we get
  \begin{align*}
    L^r(w;Y) &= L^r\hab{w;[X,Y_k]_{\theta_k}} = \bracb{L^r(v;X),L^r(w_k;Y_k)}_{\theta_k} \\ &= L^r\hab{v;X}^{1-\theta_k} \cdot L^{r}\hab{w_k;Y_k}^{\theta_k} \\
    &= \has{L^{r}\hab{w_1;Y_1}^{\tilde{\theta}_1} \cdots L^{r}\hab{w_{k-1};Y_{k-1}}^{\tilde{\theta}_{k-1}}}^{1-\theta_k} \cdot \hab{L^{r}\hab{w_k;Y_k}}^{\theta_k} \\
    &= L^{r}\hab{w_1;Y_1}^{\theta_1} \cdots L^{r}\hab{w_k;Y_k}^{\theta_k},
  \end{align*}
  which proves the lemma.
\end{proof}

With these preparatory lemmata we are now ready to prove our main theorem. We first state and prove the result in terms of $(m+1)$-tuples of functions. Afterwards, we present the main result, Theorem \ref{thm:maincor}, as a corollary. We write $\inc^{j=1,\ldots,m}_{a_j,b_j,\cdots}$ for a non-decreasing function $[1,\infty)^{2m} \to [1,\infty)$ depending on the parameters $a_j,b_j,\cdots$ for $j=1,\cdots,m$ and the dimension $d$.
\begin{theorem}[Multilinear limited range  extrapolation for vector-valued functions]\label{thm:multi-limited-range-extrap}
Fix $m\in\N$, let $X_1,\ldots,X_m$ be quasi-Banach function spaces over a $\sigma$-finite measure space $(S,\mu)$ and define $X=X_1\cdots X_m$. Let
\[
\mc{F} \subseteq L^0_+(\R^d;X)\times L^0_+(\R^d;X_1)\times\cdots\times L^0_+(\R^d;X_m).
\]
For $j=1,\cdots,m$ fix $0<p^-_j<p^+_j\leq\infty$ and assume that $X_j\in\UMD_{p_j^-,p_j^+}$. Moreover assume that for all $p_j\in(p_j^-,p_j^+)$, weights $w_j^{p_j}\in A_{p_j/p_j^-}\cap\RH_{(p^+_j/p_j)'}$ and $(f,f_1,\ldots,f_m)\in\mc{F}$, we have
\begin{equation}\label{eqn:extrap-assn2}
\nrm{f(\cdot,s)}_{L^p(w^p)} \leq \inc^{j=1,\ldots,m}_{p_j,p_j^-,p_j^+} \bigl([w_j^{p_j}]_{A_{p_j/p_j^-}},[w_j^{p_j}]_{\RH_{(p^+_j/p_j)'}}\bigr) \prod_{j=1}^m\nrm{f_j(\cdot,s)}_{L^{p_j}(w_j^{p_j})}
\end{equation}
for $\mu$-a.e. $s \in S$, where $\frac{1}{p}=\sum_{j=1}^m\frac{1}{p_j}$ and $w=\prod_{j=1}^m w_j$.

Then for all $p_j\in(p_j^-,p_j^+)$, weights $w_j^{p_j}\in A_{p_j/p_j^-}\cap\RH_{(p^+_j/p_j)'}$, and $(f,f_1,\ldots,f_m)\in\mc{F}$, we have
\begin{equation}\label{eqn:extrap-goal2}
\nrm{f}_{L^p(w^p;X)} \leq \inc^{j=1,\ldots,m}_{X_j,p_j,p_j^-,p_j^+} \bigl([w_j^{p_j}]_{A_{p_j/p_j^-}},[w_j^{p_j}]_{\RH_{(p^+_j/p_j)'}}\bigr) \prod_{j=1}^m\nrm{f_j}_{L^{p_j}(w_j^{p_j};X_j)},
\end{equation}
with $w$ and $p$ as before.
\end{theorem}
\begin{proof}
We split the proof in two steps. In the first step we show that the conclusion of the theorem holds for specific choices of $p_j\in(p_j^-,p_j^+)$. In the second step we conclude that the result holds for all $p_j\in(p_j^-,p_j^+)$ through scalar-valued extrapolation.

\textbf{Step 1:} Let $1 < \beta < \min_j \limits \frac{p_j^+}{p_j^-}$. We will first prove the theorem for $p_j := \beta \cdot p_j^-$. Let $(f,f_1,\ldots,f_m) \in \mc{F}$ and take weights $w_j^{p_j}\in A_{p_j/p_j^-}\cap\RH_{(p^+_j/p_j)'}$ for $j=1,\cdots,m$. From Theorem \ref{thm:UMDopen} and Lemma \ref{lem:muckenhoupt}\ref{it:mw3} it follows that there exists an $1<\alpha < \beta$ such that
\begin{equation}\label{eqn:qprop}
  X_j\in \UMD_{\alpha p_j^-,p_j^+} \qquad \text{and} \qquad  w_j^{p_j}\in A_{p_j/(\alpha p_j^-)}\cap\RH_{\ha*{{p_j^+/p_j}}'}
\end{equation}
with $[w_j^{p_j}]_{A_{p_j/(\alpha p_j^-)}}\leq C_{p_j,p_j^-}\,[w_j^{p_j}]_{A_{p_j/p_j^-}}$ for all $j\in\{1,\ldots, m\}$. We define
\begin{align*}
  q_j:=\alpha p_j^- \qquad &\text{and} \qquad q:=\frac{\alpha}{\beta} p .
\intertext{Note that}
  \frac{p}{q} = \frac{\beta}{\alpha} = \frac{p_j}{q_j} \qquad &\text{and} \qquad \frac{1}{q} = \frac{\beta}{\alpha} \sum_{j=1}^m \frac{1}{p_j} = \sum_{j=1}^m \frac{1}{q_j}.
\end{align*}

Let $u_j\in L^{(p_j/q_j)'}(w_j^{p_j};(X_j^{q_j})^\ast)$. By Proposition \ref{prop:UMDp+p-} and \eqref{eqn:qprop} we may apply Lemma \ref{lem:rubioalgoritme} for $j=1,\cdots,m$ with
\begin{equation*}
  r=p_j/q_j, \qquad r_+=p_j^+/q_j, \qquad Y=(X_j^{q_j})^\ast,
\end{equation*}
and weight $w_j^{p_j}$ to find nonnegative $v_j \in L^{(p_j/q_j)'}(w_j^{p_j};(X^{q_j})^\ast)$ such that
 \begin{itemize}
   \item $u_j\leq v_j$.
   \item $\nrm{v_j}_{L^{(p_j/q_j)'}(w_j^{p_j};(X_j^{q_j})^\ast)} \leq 2$.
   \item $v_j(\cdot,s)w_j^{p_j} \in A_{q_j/p_j^-} \cap RH_{(p_j^+/q_j)'}$ with for $\mu$-a.e. $s \in S$
   \begin{equation*}
   \max\cbraceb{[v_j(\cdot,s)w_j^{p_j}]_{A_{q_j/p_j^-}},[v_j(\cdot,s)w_j^{p_j}]_{\RH_{(p_j^+/q_j)'}} } \leq \inc_{X_j,p_j,p_j^-,p_j^+}\hab{[w_j^{p_j}]_{A_{p_j/p_j^-}}, [w_j^{p_j}]_{RH_{(p_j^+/p_j)'}}}.
\end{equation*}
 \end{itemize}
We set $v=\prod_{j=1}^m v_j^{1/q_j}w_j^{p_j/q_j}$ so that
\begin{equation*}
  \has{\prod_{j=1}^m u_j^{p/p_j}}w^p\leq\prod_{j=1}^m v_j^{q/q_j}w_j^{qp_j/q_j}=v^q.
\end{equation*}

Let $(f,f_1,\ldots,f_m)\in\mc{F}$. By Fubini's theorem, H\"older's inequality, the assumption \eqref{eqn:extrap-assn2}, and the properties of the $v_j$ we have
\begin{equation}\label{eqn:dcompest}
\begin{split}
\int_{\R^d}&\int_S  f^q \prod_{j=1}^m u_j^{p/p_j}  \dd\mu \, w^p\dd x \leq\int_S  \int_{\R^d} f^q v^q  \dd x \, \dd \mu\\
&\leq \int_S \inc^{j=1,\ldots,m}_{q_j,p_j^-,p_j^+} \bigl([v_j(\cdot,s)w_j^{p_j}]_{A_{q_j/p_j^-}},[v_j(\cdot,s)w_j^{p_j}]_{\RH_{(p_j^+/q_j)'}}\bigr)\prod_{j=1}^m\|f_j(\cdot,s)\|^q_{L^{q_j}(v_j w_j^{p_j})} \dd \mu(s) \\
&\leq\inc^{j=1,\ldots,m}_{X_j,p_j,p_j^-,p_j^+}\hab{[w_j^{p_j}]_{A_{p_j/p_j^-}}, [w_j^{p_j}]_{RH_{(p_j^+/p_j)'}}}\prod_{j=1}^m\left(\int_S\int_{\R^d} f_j^{q_j}v_j w_j^{p_j}\,\dd x\,\dd\mu\right)^{q/q_j}\\
&\leq\inc^{j=1,\ldots,m}_{X_j,p_j,p_j^-,p_j^+}\hab{[w_j^{p_j}]_{A_{p_j/p_j^-}}, [w_j^{p_j}]_{RH_{(p_j^+/p_j)'}}}\prod_{j=1}^m\|f_j^{q_j}\|^{q/q_j}_{L^{p_j/q_j}(w_j^{p_j};X_j^{q_j})}\|v_j\|^{q/q_j}_{L^{(p_j/q_j)'}(w_j^{p_j};(X_j^{q_j})^\ast)}\\
&\leq\inc^{j=1,\ldots,m}_{X_j,p_j,p_j^-,p_j^+}\hab{[w_j^{p_j}]_{A_{p_j/p_j^-}}, [w_j^{p_j}]_{RH_{(p_j^+/p_j)'}}} \has{\prod_{j=1}^m\|f_j\|_{L^{p_j}(w_j^{p_j};X_j)}\|u_j\|^{1/q_j}_{L^{(p_j/q_j)'}(w_j^{p_j};(X_j^{q_j})^\ast)}}^q.
\end{split}
\end{equation}

Now by Lemma \ref{lemma:pBFSbochner} with
\begin{equation*}
 r= p/q, \qquad  Y_j = X_j^{q_j}, \qquad \theta_j = q / q_j
\end{equation*}
and weights $w_j^{p_j}$, Proposition \ref{prop:pBFS}\ref{it:pBFSdual} and the duality of Bochner spaces (see \cite[Corollary 1.3.22]{HNVW16}), we have
\begin{equation*}
  L^{p/q}(w^p;X^q)^* = L^{(p_1/q_1)'}\hab{w_1^{p_1};(X_1^{q_1})^*}^{q/q_1} \cdots L^{(p_m/q_m)'}\hab{w_m^{p_m};(X_m^{q_m})^*}^{q/q_m}.
\end{equation*}
Thus, picking $u\in L^{p/q}(w^p;X^q)^*$ of norm $1$, by taking an infimum over all decompositions $u=\prod_{j=1}^m u_j^{p/p_j}$ with $u_j\in L^{(p_j/q_j)'}\hab{w_j^{p_j};(X_j^{q_j})^\ast}$, we may conclude from \eqref{eqn:dcompest} that
\[
\int_{\R^d}\int_S  f^q u  \dd\mu \, w^p\dd x\leq\inc^{j=1,\ldots,m}_{X_j,p_j,p_j^-,p_j^+}\hab{[w_j^{p_j}]_{A_{p_j/p_j^-}}, [w_j^{p_j}]_{RH_{(p_j^+/p_j)'}}} \has{\prod_{j=1}^m\|f_j\|_{L^{p_j}(w_j^{p_j};X_j)}}^q.
\]
Thus, the result for these specific $p_j$'s follows from
\begin{align*}
  \|f\|^q_{L^p(w^p;X)} = \|f^q\|_{L^{p/q}(w^p;X^q)}= \sup_{ \|u\|_{L^{p/q}(w^p;X^q)^*}=1} \int_{\R^d}\int_S f^qu\dd\mu\,w^p\dd x.
\end{align*}
\bigskip

{\bf Step 2:} We may finish the proof for general $p_j$'s by appealing to the scalar-valued limited range multilinear extrapolation result by Cruz-Uribe and Martell \cite{CM17}. Indeed, we define a new family
\[
\widetilde{\mc{F}}:=\cbraceb{\hab{\|f\|_{X},\|f_1\|_{X_1},\ldots,\|f_m\|_{X_m}} : (f,f_1,\ldots,f_m) \in \mc{F}}.
\]
Then $\widetilde{\mc{F}} \subset L^0_+(\R^d)^{m+1}$ and by Step 1 we have
\[
\nrmb{\tilde{f}}_{L^p(w^p)}\leq\inc^{j=1,\ldots,m}_{X,p_j^-,p_j^+}\hab{[w_j^{p_j}]_{A_{p_j/p_j^-}}, [w_j^{p_j}]_{RH_{(p_j^+/p_j)'}}} \prod_{j=1}^m\nrmb{\tilde{f}_j}_{L^{p_j}(w_j^{p_j})}
\]
for certain $p_j\in (p_j^-,p_j^+)$, all $(\tilde{f},\tilde{f}_1,\ldots,\tilde{f}_m)\in\widetilde{\mc{F}}$, and all weights $w_j^{p_j}\in A_{p_j/p_j^-}\cap\RH_{(p_j^+/p_j)'}$. The result for general $p_j\in (p_j^-,p_j^+)$ then follows directly from \cite[Theorem 1.3 and Corollary 1.11]{CM17}, proving the assertion.
\end{proof}
Finally, we will prove the main result from the introduction, which is a direct corollary of Theorem \ref{thm:multi-limited-range-extrap}.
\begin{proof}[Proof of Theorem \ref{thm:maincor}]
We wish to apply Theorem \ref{thm:multi-limited-range-extrap} to the collection
\begin{equation*}
\mc{F} = \cbraceb{(\abs{\widetilde{T}(f_1,\ldots,f_m)},\abs{f_1},\ldots,\abs{f_m}): \,\,f_j\colon\R^d \to X_j \text{ simple}}.
\end{equation*}
Our assumption implies that there are $p_j\in(p_j^-,p_j^+)$ so that for all weights $w_j^{p_j}\in A_{p_j/p_j^-}\cap\RH_{(p_j^+/p_j)'}$ the a priori estimate \eqref{eqn:extrap-assn2} in Theorem \ref{thm:multi-limited-range-extrap} holds. By appealing to the scalar-valued limited range multilinear extrapolation result \cite{CM17} we may conclude that \eqref{eqn:extrap-assn2} in fact holds for all $p_j\in(p_j^-,p_j^+)$ and weights $w_j^{p_j}\in A_{p_j/p_j^-}\cap\RH_{(p_j^+/p_j)'}$. Thus, Theorem \ref{thm:multi-limited-range-extrap} implies that
\begin{equation}\label{eq:extend}
  \nrmb{\widetilde{T}(f_1,\ldots,f_m)}_{L^p(w^p;X)} \leq C\,\prod_{j=1}^m\|f_j\|_{L^{p_j}(w_j^{p_j};X_j)}
\end{equation}
for all simple functions $f_j\colon\R^d \to X_j$, where $C$ depends only on the $X_j$, $p_j$, and the characteristic constants of the weights. If $T$ is $m$-linear, then \eqref{eq:extend} extends directly to all $f_j\in L^{p_j}(w_j^{p_j};X_j)$ by density. If $T$ is $m$-sublinear and positive valued, then we fix simple functions $f_j\colon:\R^d \to X_j$ for $j\in\{2,\ldots,m\}$. For any pair of simple functions $f_1,g_1:\R^d\to X_1$ we have
\begin{equation*}
  \widetilde{T}(f_1,\ldots, f_m) = \widetilde{T}(f_1-g_1+g_1,f_2\ldots, f_m) \leq \widetilde{T}(f_1-g_1,f_2\ldots, f_m)+\widetilde{T}(g_1,f_2\ldots,f_m)
\end{equation*}
so that
\begin{align*}
  \nrmb{\widetilde{T}(f_1,\ldots, f_m)-\widetilde{T}(g_1,f_2\ldots,f_m)}_{L^p(w^p;X)} &\leq \nrmb{\widetilde{T}(f_1-g_1,f_2\ldots, f_m)}_{L^p(w^p;X)}\\
  &\leq C\, \|f_1-g_1\|_{L^{p_1}(w_1^{p_1};X_1)}\prod_{j=2}^m\|f_j\|_{L^{p_j}(w_j^{p_j};X_j)}.
\end{align*}
Thus, \eqref{eq:extend} extends to arbitrary $f_1 \in L^{p_1}(w_1^{p_1};X_1)$ by density. Iterating this argument for $j=2,\ldots m$ proves the result.
\end{proof}

\section{Applications}\label{sec:applications}
In this section we apply our main result to various operators, for which we obtain new vector-valued bounds.
\subsection{The bilinear Hilbert transform}
For $d=1$, The bilinear Hilbert transform $\BHT$ is defined by
\[
\BHT(f,g)(x)=\pv\int_{\R}\!f(x-t)g(x+t)\frac{\dd t}{t}.
\]
After its initial introduction by Calder\'on, it took thirty years until $L^p$ estimates were established by Lacey and Thiele \cite{LT99}. They showed that for $p_1,p_2\in (1,\infty]$ with $\frac{1}{p}=\frac{1}{p_1}+\frac{1}{p_2}<\frac{3}{2}$ one has
\begin{equation}\label{eqn:bihilisbounded}
\|\BHT(f,g)\|_{L^p}\leq C\|f\|_{L^{p_1}}\|g\|_{L^{p_2}}.
\end{equation}
As for weighted bounds, the first results were obtained by Culiuc, di Plinio, and Ou \cite{CDPO16}, and through the extrapolation result of Cruz-Uribe and Martell the range of exponents was increased \cite{CM17}, in particular recovering the full range of exponents for the unweighted result \eqref{eqn:bihilisbounded}. It was already shown in \cite{CM17} that this result implies corresponding vector-valued bounds for $\BHT$ for certain $\ell^s$-spaces. Moreover, vector-valued bounds for $\BHT$ have also been considered by Benea and Muscalu \cite{BM16}. In particular, they consider functions taking values in iterated $L^s$-spaces, see \cite[Theorem 8]{BM16} including the case $s=\infty$.

Through our main result we are able to obtain a new bounded vector-valued extension of the bilinear Hilbert transform. By combining the weighted estimates in \cite[Theorem 1.18]{CM17} with Theorem \ref{thm:maincor}, we get:
\begin{theorem}\label{thm:bilhilbo}
Let $q_1,q_2\in(1,\infty)$ so that $\frac{1}{q_1}+\frac{1}{q_2}<1$. For $j\in\{1,2\}$, define
\[
p_j^-:=\frac{2q_j}{1+q_j},\qquad p_j^+:=2q_j.
\]
Let $X=X_1\cdot X_2$, where $X_1$, $X_2$ are quasi-Banach function spaces over a $\sigma$-finite measure space $(S,\mu)$ satisfying $X_j\in\UMD_{p_j^-,p_j^+}$.
Then for all $p_1$, $p_2$ with $p_j\in(p_j^-,p_j^+)$ and all weights $w_1$, $w_2$ satisfying $w_j^{p_j}\in A_{p_j/p_j^-}\cap\RH_{(p_j^+/p_j)'}$ we have
\[
\big\|\widetilde{\BHT}(f,g)\big\|_{L^p(w^p;X)}\leq C'
\|f\|_{L^{p_1}(w_1^{p_1};X_1)}\|g\|_{L^{p_2}(w_2^{p_2};X_2)}
\]
for all $f\in L^{p_1}(w_1^{p_1};X_1)$, $g\in L^{p_2}(w_2^{p_2};X_2)$, where $\frac{1}{p}=\frac{1}{p_1}+\frac{1}{p_2}$, $w=w_1w_2$, and where $C'>0$ depends only on the $X_j$, $p_j$, $q_j$, and the characteristic constants of the weights.
\end{theorem}

By Example \ref{ex:exproducts} we have $\ell^s = \ell^{s_1} \cdot \ell^{s_2}$ for $s_1, s_2\in (0,\infty)$ and $\frac{1}{s}=\frac{1}{s_1}+\frac{1}{s_2}$. Thus, we recover \cite[Theorem 1.29]{CM17} by Example \ref{ex:concavification}. It is implicit from the arguments in \cite{CDPO16} that there are more general weighted estimates for $\BHT$ leading to a wider range of vector-valued extensions. For a technical discussion on this, we refer the reader to \cite[Section 5]{CM17}.

Furthermore by Proposition \ref{prop:UMDp+p-}\ref{it:UMDLp} we can also handle iterated $L^s$-spaces as considered by Benea and Muscalu \cite{BM16}, but our results do not overlap as we do not obtain bounds involving $L^\infty$-spaces. Such spaces might be in the scope of a generalized version of our main theorem using multilinear weight classes combined with a multilinear $\UMD$ condition, see also Remark \ref{rem:intro} and \cite{N18}.

Finally, we mention the vector-valued bounds obtained by Hyt\"onen, Lacey, and Parissis \cite{HLP13} for the related bilinear quartile operator (the Fourier-Walsh model of $\BHT$). They consider estimates involving triples of more general $\UMD$ Banach spaces with so called \emph{quartile type} $q$.  It is unknown whether these estimates hold for $\BHT$ itself. Note that a Banach function space $X \in \UMD_{p_-,p_+}$ has quartile type $\max\cbrace{p_-',p_+}$ by Proposition \ref{prop:UMDp+p-}\ref{it:UMDinterp} and \cite[Proposition 4.1]{HLP13}.
\subsection{Multilinear Calder\'on-Zygmund operators}
Let $T$ be an $m$-linear operator, initially defined for $m$-tuples $f_1,\ldots,f_m\in C_c^\infty(\R^d)$, that satisfies
\[
T(f_1,\ldots, f_m)(x)=\int_{(\R^d)^m}\!K(x,y_1,\ldots,y_m)\prod_{j=1}^m f_j(y_j)\dd y,
\]
whenever $x\notin\cap_{j=1}^m\supp f_j$, where $K$ is a kernel defined in $(\R^d)^{m+1}$ outside of the diagonal $y_0=y_1=\cdots=y_m$. If $K$ satisfies the estimate
\[
|\partial_{y_0}^{\alpha_0}\cdots\partial_{y_m}^{\alpha_m} K(y_0,\ldots,y_m)|\leq\frac{C}{\left(\sum_{j,k=0}^m|y_j-y_k|\right)^{md+|\alpha_1|+\cdots+|\alpha_m|}}
\]
for all multi-indices $\alpha_j$ so that $\sum_{j=1}^m|\alpha_j|\leq 1$ and if there exist $p_1,\ldots, p_m$ so that $T$ extends to a bounded operator $L^{p_1}\times\cdots L^{p_m}\to L^p$ with $\frac{1}{p}=\sum_{j=1}^m\frac{1}{p_j}$, then $T$ is called an $m$-linear Calder\'on-Zygmund operator.

Multilinear Calder\'on-Zygmund operators first appeared in the work \cite{CM75} by Coifman and Meyer. Weighted estimates for these operators have been considered for example by Grafakos and Torres in \cite{GT02} and subsequently by Grafakos and Martell in \cite{GM04}, where it was shown that for all $p_j\in(1,\infty)$, all weights $w_j^{p_j}\in A_{p_j}$, and all $f_j\in L^{p_j}(w_j^{p_j})$ we have
\begin{equation*}
  \nrmb{T(f_1,\ldots,f_m)}_{L^p(w^p)} \leq C\,\prod_{j=1}^m\|f_j\|_{L^{p_j}(w_j^{p_j})},
\end{equation*}
where $w=\prod_{j=1}^m w_j$ and $\frac{1}{p}=\sum_{j=1}^m\frac{1}{p_j}$, and where $C$ depends only on the characteristic constants of the weights. Thus, by Theorem \ref{thm:maincor} we obtain the following result:
\begin{theorem}
Let $T$ be an $m$-linear Calder\'on-Zygmund operator and suppose $X_1,\ldots,X_m\in\UMD$. Then for all $p_j\in(1,\infty)$, all weights $w_j^{p_j}\in A_{p_j}$, and all $f_j\in L^{p_j}(w_j^{p_j};X_j)$ we have
\begin{equation*}
  \nrmb{\widetilde{T}(f_1,\ldots,f_m)}_{L^p(w^p;X)} \leq C'\,\prod_{j=1}^m\|f_j\|_{L^{p_j}(w_j^{p_j};X_j)},
\end{equation*}
where $X=X_1\cdots X_m$, $w=\prod_{j=1}^m w_j$, $\frac{1}{p}=\sum_{j=1}^m\frac{1}{p_j}$, and where $C'$ depends only on the $X_j$, $p_j$, and the characteristic constants of the weights.
\end{theorem}
This result is new, as previously only $\ell^s$-valued extensions had been considered in \cite{GM04}.

We wish to point out that, using a more appropriate multilinear weight condition, more general weighted bounds for multilinear Calder\'on-Zygmund operators have been found in \cite{LOPTT09}.
\subsection{Limited range extrapolation}
The remaining examples are for the linear case $m=1$.
\begin{example}[Fourier multipliers I]
For $a<b$, $q\in[1,\infty)$, and a function $m:[a,b]\to\C$ we define the $q$-variation norm
\[
\nrm{m}_{V^q([a,b])}:=\|m\|_{L^\infty([a,b])}+\sup\left(\sum_{j=0}^{n-1}|m(t_{j+1})-m(t_j)|^q\right)^{\frac{1}{q}},
\]
where the supremum is taken over all partitions $a=t_0<\ldots<t_n=b$ of the interval $[a,b]$. Let $\ms{D}:=\big\{\pm(2^k,2^{k+1}]:k\in\Z\big\}$ be the dyadic decomposition of $\R$. Then we define a class of multipliers
\[
V^q(\ms{D}):=\big\{m:\R\to\C : \sup_{I\in\ms{D}}\nrm{m|_{I}}_{V^q(I)}<\infty\big\}.
\]

For $q>2$ and $p_+:=2\left(\frac{q}{2}\right)'$
it was shown by Kr\'{o}l \cite[Theorem A(ii)]{Kr14} that for all $p\in[2,p_+)$, $w\in A_{p/2}\cap\RH_{(p_+/p)'}$ and $m\in V^q(\ms{D})$ the Fourier multiplier $T_m$ defined by $\ms{F}(T_mf)=m\ms{F}f$ satisfies
\[
\|T_m\|_{L^p(w)\to L^p(w)}<\infty.
\]
Therefore one may readily apply Theorem \ref{thm:maincor} with $p_-=2$ to the linear operator $T_m$. So for any Banach function space $X$ such that $X\in\UMD_{2,p_+}$ we find for all $p\in (2,p_+)$, all $w\in A_{p/2}\cap\RH_{(p_+/p)'}$ and $m\in V^q(\ms{D})$ that
\[
\nrmb{\widetilde{T}_m}_{L^p(w;X)\to L^p(w;X)}<\infty.
\]
Note that \cite[Theorem A(i)]{Kr14} was already extrapolated to the vector-valued setting by Amenta, Veraar, and the first author in \cite{ALV17}, proving that for
$m\in V^q(\ms{D})$ with $q\in[1,2]$ the Fourier multiplier $T_m$ has a bounded vector-valued extension for Banach function spaces $X\in\UMD_{q,\infty}$. Furthermore extensions of \cite[Theorem A]{Kr14} for operator-valued Fourier multipliers have been obtained in \cite{ALV18}.
\end{example}

\begin{example}[Riesz transforms associated with elliptic operators]
Let $A\in L^\infty(\R^d;\C^{d\times d})$ satisfy an ellipticity condition $\re(A(x)\xi\cdot\overline{\xi})\geq\lambda|\xi|^2$ for a.e. $x\in\R^d$, and all $\xi\in\C^d$. Then we may consider a second order divergence form operator
\[
L:=-\divv(A\nabla f),
\]
defined on $L^2$, which due to the ellipticity condition on $A$ generates an analytic semigroup $(e^{-t L})_{t>0}$ in $L^2$. Let $1\leq p_-<p_+\leq \infty$. If both the semigroup and the gradient family $(\sqrt{t}\nabla e^{-t L})_{t>0}$ satisfy $L^{p_-}$--$L^{p_+}$ off-diagonal estimates, then the Riesz transform $R:=\nabla L^{-1/2}$ is a bounded operator in $L^p(w)$ for all $p\in(p_-,p_+)$ and all weights $w\in A_{p/p_-}\cap\RH_{(p_+/p)'}$, see \cite{AM07, BFP16}.  The values of $p_-$ and $p_+$ for which such off-diagonal estimates hold depend on the dimension $d$ and on the matrix-valued function $A$ and are studied in detail in \cite{Au07}. The result we obtain is that if a Banach function space $X$ satisfies $X\in\UMD_{p_-,p_+}$, then for all $p\in (p_-,p_+)$ and all weights $w\in A_{p/p_-}\cap\RH_{(p_+/p)'}$ we have
\[
\nrmb{\widetilde{R}}_{L^p(w;X)\to L^p(w;X)}<\infty.
\]
This result is new in the sense that previously such bounds were previously only known for $X=\ell^s$ through the limited range extrapolation result in \cite{AM07}.
\end{example}

Next, we consider a class of operators satisfying a certain sparse domination property. A collection $\mathcal{S}$ of cubes in $\R^d$ is called \textit{sparse} if there is a pairwise disjoint collection of sets $(E_Q)_{Q\in\mathcal{S}}$ so that for each $Q\in\mathcal{S}$ we have $E_Q\subseteq Q$ and $|Q|\leq 2|E_Q|$. We say that a (sub)linear operator $T$ satisfies the sparse domination property with parameters $1\leq p_-<p_+\leq\infty$ if there is a $C>0$ so that for all compactly supported smooth functions $f,g:\R^d\to\C$ we have
\begin{equation}\label{eq:sparsebounds}
|\langle Tf,g\rangle|\leq C\sup_{\mathcal{S}\text{ sparse}}\sum_{Q\in\mathcal{S}}\langle |f|^{p_-}\rangle_{Q}^{\frac{1}{p_-}}\langle |g|^{p_+}\rangle_{Q}^{\frac{1}{p_+}}|Q|,
\end{equation}
where the supremum runs over all sparse collections of cubes $\mc{S}$. For an operator $T$ we denote the optimal constant $C$ appearing in \eqref{eq:sparsebounds} by $\|T\|_{S(p_-,p_+)}$. Estimates in the form \eqref{eq:sparsebounds} were first considered in \cite{BFP16} where it was shown that
\[
\|T\|_{S(p_-,p_+)}<\infty\,\Rightarrow\,\|T\|_{L^p(w)\to L^p(w)}<\infty
\]
for $p \in (p_-,p_+)$ and $w\in A_{p/p_-}\cap\RH_{(p_+/p)'}$ by giving a quantitative estimate in terms of the characteristic constants of the weight. Thus, we may readily apply Theorem \ref{thm:maincor} with $m=1$ to any linear or positive-valued sublinear $T$ such that $\|T\|_{S(p_-,p_+)}<\infty$. This yields the following result:
\begin{theorem}\label{thm:sparsevecext}
Let $T$ be a linear or a positive-valued sublinear operator and let $X$ be a Banach function space over a $\sigma$-finite measure space $(S,\mu)$. Assume that for all simple functions $f:\R^d\to X$ the function $\widetilde{T}f(x,s):=T(f(\cdot,s))(x)$ is well-defined and strongly measurable.

If there are $1\leq p_-<p_+\leq\infty$ such that
\[
X\in\UMD_{p_-,p_+},\qquad \|T\|_{S(p_-,p_+)}<\infty,
\]
then for all $p\in(p_-,p_+)$, all weights $w\in A_{p/p_-}\cap\RH_{(p_+/p)'}$, and all $f\in L^p(w;X)$, we have
\begin{equation}\label{eqn:sparseconc}
\nrmb{\widetilde{T}f}_{L^p(w;X)}\leq C\|f\|_{L^p(w;X)},
\end{equation}
where $C$ depends only on $X$, $p$, $p_-$, $p_+$, and the characteristic constants of $w$.
\end{theorem}
We emphasize again that if $T$ is linear, then $\widetilde{T}f$ is automatically well-defined and strongly measurable for any simple function $f:\R^d\to X$, see also Remark \ref{rem:mainthm}. We conclude this section by giving several examples of operators satisfying sparse bounds.
\begin{example}[Fourier multipliers II]
For each $\delta\geq 0$, the Bochner-Riesz multiplier $B_\delta$ is defined as the Fourier multiplier $\ms{F}(B_\delta f)=(1-|\xi|^2)_+^\delta\ms{F}f$, where $t_+=\max(t,0)$. For $\delta\geq (d-1)/2$, $B_\delta$ satisfies weighted bounds $\|B_\delta\|_{L^p(w)\to L^p(w)}<\infty$ for any $p\in (1,\infty)$ and any $w\in A_p$, see \cite{Bu93, DR86, SS92}.

The situation is more complicated when $0<\delta<(d-1)/2$ and weighted bounds for such $\delta$ have, for example, been considered in \cite{CDL12, Ch85, DMOS08}. The idea to quantify weighted bounds for $B_\delta$ for $0<\delta<(d-1)/2$ through sparse domination was initiated by Benea, Bernicot, and Luque \cite{BBL17}. It was shown by Lacey, Mena, and Reguera that for this range of $\delta$ there are explicit subsets $R_{\delta,d}$ of the plane so that
\[
\\|B_\delta\|_{S(p_-,p_+)}<\infty
\]
for $(p_-,p_+)\in R_{\delta,d}$, see \cite{LMR17}. We also refer the reader to the recent work by Kesler and Lacey \cite{KL17} containing certain sparse endpoint bounds in dimension $d=2$.

As far as we know, the only vector-valued estimates that have been shown for $B_\delta$ have been for $X=\ell^s$, see \cite{BBL17}. For any $p_-$, $p_+$ and $\delta$ for which $\|B_\delta\|_{S(p_-,p_+)}<\infty$, we obtain by Theorem \ref{thm:sparsevecext} that inequality \eqref{eqn:sparseconc} with $\widetilde{T} = \widetilde{B}_\delta$ holds for any Banach function space $X$ satisfying $X\in\UMD_{p_-,p_+}$, yielding new vector-valued estimates.
\end{example}
\begin{example}[Spherical maximal operators]
Let $(S^{d-1},\sigma)$ denote the unit sphere in $\R^d$ equipped with its normalized Euclidean surface measure $\sigma$. For a smooth function $f$ on $\R^d$ we denote by $A_r f(x)$ the average of $f$ over the sphere centered at $x$ of radius $r>0$, i.e.,
\[
A_rf(x):=\int_{S^{d-1}}f(x-r\omega)\dd\sigma(\omega).
\]
We respectively define the lacunary spherical maximal operator and the full spherical maximal operator by
\[
M_{\lac}f:=\sup_{k\in\Z}|A_{2^k}f|,\quad M_{\full}f:=\sup_{r>0}|A_r f|,
\]
the latter having been introduced by Stein \cite{St76} and the former having been studied by Calder\'{o}n \cite{Ca79}. It was shown by Lacey \cite{La17} that for explicit subsets $L_d$, $F_d$ of the plane we have
\begin{align*}
\|M_{\lac}\|_{S(p_-,p_+)}<\infty & ,\quad\text{for $(p_-,p_+)\in L_d$,}\\
\|M_{\full}\|_{S(p_-,p_+)}<\infty & ,\quad\text{for $(p_-,p_+)\in F_d$.}
\end{align*}
These results recover the previous known $L^p$-bounds for these operators and yield weighted bounds.

To apply Theorem \ref{thm:sparsevecext} to $M_{\lac}$ and $M_{\full}$, one needs to check that these operators have well-defined and strongly measurable extensions to $X$-valued simple functions with $X$ a Banach function space over $(S,\mu)$. This can be checked as in \cite[Lemma 3.1]{HL17}. Therefore it follows from Theorem \ref{thm:sparsevecext} that if $(p_-,p_+)\in L_d$ or $(p_-,p_+)\in F_d$, then for any Banach function space $X\in\UMD_{p_-,p_+}$ we obtain the bound \eqref{eqn:sparseconc} for $\widetilde{T}=\widetilde{M}_{\lac}$ or $\widetilde{T}=\widetilde{M}_{\full}$ respectively. As far as we know, this is the first instance that vector-valued extensions have been considered for these operators.
\end{example}

\bibliographystyle{plain}

\begin{thebibliography}{10}

\bibitem{ALV18}
A.~Amenta, E.~Lorist, and M.~C. Veraar.
\newblock Fourier multipliers in {B}anach function spaces with {UMD}
  concavifications.
\newblock  {\em Trans. Amer. Math. Soc.}, 371(7):4837--4868, 2019.

\bibitem{ALV17}
A.~Amenta, E.~Lorist, and M.~C. Veraar.
\newblock Rescaled extrapolation for vector-valued functions.
\newblock {\em Publ. Mat.}, 63(1):155--182, 2019.

\bibitem{Au07}
P.~Auscher.
\newblock On necessary and sufficient conditions for {$L^p$}-estimates of
  {R}iesz transforms associated to elliptic operators on {$\mathbb{R}^n$} and
  related estimates.
\newblock {\em Mem. Amer. Math. Soc.}, 186(871):xviii+75, 2007.

\bibitem{AM07}
P.~Auscher and J.~M. Martell.
\newblock Weighted norm inequalities, off-diagonal estimates and elliptic
  operators. {I}. {G}eneral operator theory and weights.
\newblock {\em Adv. Math.}, 212(1):225--276, 2007.

\bibitem{BBL17}
C.~Benea, F.~Bernicot, and T.~Luque.
\newblock Sparse bilinear forms for {B}ochner {R}iesz multipliers and
  applications.
\newblock {\em Trans. London Math. Soc.}, 4(1):110--128, 2017.

\bibitem{BM16}
C.~Benea and C.~Muscalu.
\newblock Multiple vector-valued inequalities via the helicoidal method.
\newblock {\em Anal. PDE}, 9(8):1931--1988, 2016.

\bibitem{BFP16}
F.~Bernicot, D.~Frey, and S.~Petermichl.
\newblock Sharp weighted norm estimates beyond {C}alder\'on-{Z}ygmund theory.
\newblock {\em Anal. PDE}, 9(5):1079--1113, 2016.

\bibitem{BFR12}
J.~J. Betancor, J.~C. Fari\~na, and L.~Rodr{\'\i}guez-Mesa.
\newblock Hardy-{L}ittlewood and {UMD} {B}anach lattices via {B}essel
  convolution operators.
\newblock {\em J. Operator Theory}, 67(2):349--368, 2012.

\bibitem{Bo83}
J.~Bourgain.
\newblock Some remarks on {B}anach spaces in which martingale difference
  sequences are unconditional.
\newblock {\em Ark. Mat.}, 21(2):163--168, 1983.

\bibitem{Bo84}
J.~Bourgain.
\newblock Extension of a result of {B}enedek, {C}alder\'on and {P}anzone.
\newblock {\em Ark. Mat.}, 22(1):91--95, 1984.

\bibitem{Bu93}
S.~M. Buckley.
\newblock Estimates for operator norms on weighted spaces and reverse {J}ensen
  inequalities.
\newblock {\em Trans. Amer. Math. Soc.}, 340(1):253--272, 1993.

\bibitem{Bu87}
A.~V. Bukhvalov.
\newblock Interpolation of linear operators in spaces of vector functions and
  with a mixed norm.
\newblock {\em Sibirsk. Mat. Zh.}, 28(1):i, 37--51, 1987.

\bibitem{Bu83}
D.~L. Burkholder.
\newblock A geometric condition that implies the existence of certain singular
  integrals of {B}anach-space-valued functions.
\newblock In {\em Conference on harmonic analysis in honor of Antoni Zygmund,
  Vol. I, II (Chicago, Ill., 1981)}, Wadsworth Math. Ser., pages 270--286.
  Wadsworth, Belmont, CA, 1983.

\bibitem{Bu01}
D.~L. Burkholder.
\newblock Martingales and singular integrals in {B}anach spaces.
\newblock In {\em Handbook of the geometry of {B}anach spaces, {V}ol. {I}},
  pages 233--269. North-Holland, Amsterdam, 2001.

\bibitem{Ca64}
A.~P. Calder{\'o}n.
\newblock Intermediate spaces and interpolation, the complex method.
\newblock {\em Studia Math.}, 24:113--190, 1964.

\bibitem{Ca79}
C.~P. Calder\'on.
\newblock Lacunary spherical means.
\newblock {\em Illinois J. Math.}, 23(3):476--484, 1979.

\bibitem{CDL12}
M.~J. Carro, J.~Duoandikoetxea, and M.~Lorente.
\newblock Weighted estimates in a limited range with applications to the
  {B}ochner-{R}iesz operators.
\newblock {\em Indiana Univ. Math. J.}, 61(4):1485--1511, 2012.

\bibitem{Ch85}
M.~Christ.
\newblock On almost everywhere convergence of {B}ochner-{R}iesz means in higher
  dimensions.
\newblock {\em Proc. Amer. Math. Soc.}, 95(1):16--20, 1985.

\bibitem{CCV18}
P.~Cioica, S.~G. Cox, and M.~C. Veraar.
\newblock Stochastic integration in quasi-{B}anach spaces.
\newblock arXiv:1804.08947, 2018.

\bibitem{CM75}
R.~R. Coifman and Y.~Meyer.
\newblock On commutators of singular integrals and bilinear singular integrals.
\newblock {\em Trans. Amer. Math. Soc.}, 212:315--331, 1975.

\bibitem{CM17}
D.~Cruz-Uribe and J.~M. Martell.
\newblock Limited range multilinear extrapolation with applications to the
  bilinear {H}ilbert transform.
\newblock {\em Math. Ann.}, 371(1-2):615--653, 2018.

\bibitem{CMP04}
D.~Cruz-Uribe, J.~M. Martell, and C.~P\'erez.
\newblock Extrapolation from {$A_\infty$} weights and applications.
\newblock {\em J. Funct. Anal.}, 213(2):412--439, 2004.

\bibitem{CMP11}
D.~V. Cruz-Uribe, J.~M. Martell, and C.~P{\'e}rez.
\newblock {\em Weights, extrapolation and the theory of {R}ubio de {F}rancia},
  volume 215 of {\em Operator Theory: Advances and Applications}.
\newblock Birkh\"auser/Springer Basel AG, Basel, 2011.

\bibitem{CDPO16}
A.~{Culiuc}, F.~{Di Plinio}, and Y.~{Ou}.
\newblock Domination of multilinear singular integrals by positive sparse
  forms.
\newblock {\em J. Lond. Math. Soc. (2)}, 98(2):369--392, 2018.

\bibitem{DK17}
L.~Deleaval and C.~Kriegler.
\newblock Dunkl spectral multipliers with values in {UMD} lattices.
\newblock {\em J. Funct. Anal.}, 272(5):2132--2175, 2017.

\bibitem{DMOS08}
J.~Duoandikoetxea, A.~Moyua, O.~Oruetxebarria, and E.~Seijo.
\newblock Radial {$A_p$} weights with applications to the disc multiplier and
  the {B}ochner-{R}iesz operators.
\newblock {\em Indiana Univ. Math. J.}, 57(3):1261--1281, 2008.

\bibitem{DR86}
J.~Duoandikoetxea and J.~L. Rubio~de Francia.
\newblock Maximal and singular integral operators via {F}ourier transform
  estimates.
\newblock {\em Invent. Math.}, 84(3):541--561, 1986.

\bibitem{FG91}
D.~L. Fernandez and J.~B. Garcia.
\newblock Interpolation of {O}rlicz-valued function spaces and {U}.{M}.{D}.
  property.
\newblock {\em Studia Math.}, 99(1):23--40, 1991.

\bibitem{GMT93}
J.~Garc{\'{\i}}a-Cuerva, R.~Mac{\'{\i}}as, and J.~L. Torrea.
\newblock The {H}ardy-{L}ittlewood property of {B}anach lattices.
\newblock {\em Israel J. Math.}, 83(1-2):177--201, 1993.

\bibitem{GR85}
J.~Garc{\'{\i}}a-Cuerva and J.~L. Rubio~de Francia.
\newblock {\em Weighted norm inequalities and related topics}, volume 116 of
  {\em North-Holland Mathematics Studies}.
\newblock North-Holland Publishing Co., Amsterdam, 1985.
\newblock Notas de Matem{\'a}tica, 104.

\bibitem{Gr14a}
L.~Grafakos.
\newblock {\em Classical {F}ourier analysis}, volume 249 of {\em Graduate Texts
  in Mathematics}.
\newblock Springer, New York, third edition, 2014.

\bibitem{GM04}
L.~Grafakos and J.~M. Martell.
\newblock Extrapolation of weighted norm inequalities for multivariable
  operators and applications.
\newblock {\em J. Geom. Anal.}, 14(1):19--46, 2004.

\bibitem{GT02}
L.~Grafakos and R.~H. Torres.
\newblock Maximal operator and weighted norm inequalities for multilinear
  singular integrals.
\newblock {\em Indiana Univ. Math. J.}, 51(5):1261--1276, 2002.

\bibitem{HL17}
T.~S. H\"{a}nninen and E.~Lorist.
\newblock Sparse domination for the lattice {H}ardy--{L}ittlewood maximal
  operator.
\newblock {\em Proc. Amer. Math. Soc.}, 147(1):271--284, 2019.

\bibitem{HMS88}
E.~Harboure, R.~Mac{\'\i}as, and C.~Segovia.
\newblock Extrapolation results for classes of weights.
\newblock {\em Amer. J. Math.}, 110(3):383--397, 1988.

\bibitem{Ho14}
G.~Hong.
\newblock Banach lattice-valued $q$-variation and convexity.
\newblock arXiv:1410.1575, 2014.

\bibitem{Hu66}
R.~A. Hunt.
\newblock On {$L(p,\,q)$} spaces.
\newblock {\em Enseignement Math. (2)}, 12:249--276, 1966.

\bibitem{HLP13}
T.~P. Hyt\"onen, M.~T. Lacey, and I.~Parissis.
\newblock The vector valued quartile operator.
\newblock {\em Collect. Math.}, 64(3):427--454, 2013.

\bibitem{HNVW16}
T.~P. Hyt\"onen, J.~M. A. M.~van Neerven, M.~C. Veraar, and L.~Weis.
\newblock {\em Analysis in {B}anach Spaces. {V}olume {I}: {M}artingales and
  {L}ittlewood-{P}aley Theory}, volume~63 of {\em Ergebnisse der Mathematik und
  ihrer Grenzgebiete.}
\newblock Springer, 2016.

\bibitem{HPR12}
T.~P. Hyt\"onen, C.~P\'erez, and E.~Rela.
\newblock Sharp reverse {H}\"older property for {$A_\infty$} weights on spaces
  of homogeneous type.
\newblock {\em J. Funct. Anal.}, 263(12):3883--3899, 2012.

\bibitem{HV15}
T.~P. Hyt\"onen and A.~V. V\"ah\"akangas.
\newblock The local non-homogeneous {$Tb$} theorem for vector-valued functions.
\newblock {\em Glasg. Math. J.}, 57(1):17--82, 2015.

\bibitem{JN91}
R.~Johnson and C.~J. Neugebauer.
\newblock Change of variable results for {$A_p$}- and reverse {H}\"older {${\rm
  RH}_r$}-classes.
\newblock {\em Trans. Amer. Math. Soc.}, 328(2):639--666, 1991.

\bibitem{Ka84}
N.~J. Kalton.
\newblock Convexity conditions for non-locally convex lattices.
\newblock {\em Glasgow Math. J.}, 25(2):141--152, 1984.

\bibitem{KMM07}
N.~J. Kalton, S.~Mayboroda, and M.~Mitrea.
\newblock Interpolation of {H}ardy-{S}obolev-{B}esov-{T}riebel-{L}izorkin
  spaces and applications to problems in partial differential equations.
\newblock In {\em Interpolation theory and applications}, volume 445 of {\em
  Contemp. Math.}, pages 121--177. Amer. Math. Soc., Providence, RI, 2007.

\bibitem{KM98}
N.~J. Kalton and M.~Mitrea.
\newblock Stability results on interpolation scales of quasi-{B}anach spaces
  and applications.
\newblock {\em Trans. Amer. Math. Soc.}, 350(10):3903--3922, 1998.

\bibitem{Ka98}
A.~Kami\'nska.
\newblock Indices, convexity and concavity in {M}usielak-{O}rlicz spaces.
\newblock {\em Funct. Approx. Comment. Math.}, 26:67--84, 1998.
\newblock Dedicated to Julian Musielak.

\bibitem{KL17}
R.~Kesler and M.~T. Lacey.
\newblock Sparse endpoint estimates for {B}ochner-{R}iesz multipliers on the
  plane.
\newblock {\em Collect. Math.}, 69(3):427--435, 2018.

\bibitem{Kr14}
S.~Kr{\'o}l.
\newblock Fourier multipliers on weighted {$L^p$} spaces.
\newblock {\em Math. Res. Lett.}, 21(4):807--830, 2014.

\bibitem{La17}
M.~T. {Lacey}.
\newblock {Sparse Bounds for Spherical Maximal Functions}.
\newblock arXiv:1702.08594, 2017.

\bibitem{LMR17}
M.~T. {Lacey}, D.~{Mena}, and M.~C. {Reguera}.
\newblock {Sparse Bounds for Bochner-Riesz Multipliers}.
\newblock {\em Journal of Fourier Analysis and Applications}, 2017.

\bibitem{LT99}
M.~T. Lacey and C.~Thiele.
\newblock On {C}alder\'on's conjecture.
\newblock {\em Ann. of Math. (2)}, 149(2):475--496, 1999.

\bibitem{LOPTT09}
A.K. Lerner, S.~Ombrosi, C.~P\'erez, R.H. Torres, and R.~Trujillo-Gonz\'alez.
\newblock New maximal functions and multiple weights for the multilinear
  {C}alder\'on-{Z}ygmund theory.
\newblock {\em Adv. Math.}, 220(4):1222--1264, 2009.

\bibitem{LMO18}
K.~{Li}, J.M. {Martell}, and S.~{Ombrosi}.
\newblock {Extrapolation for multilinear Muckenhoupt classes and applications
  to the bilinear Hilbert transform}.
\newblock arXiv:1802.03338, 2018.

\bibitem{LT79}
J.~Lindenstrauss and L.~Tzafriri.
\newblock {\em Classical {B}anach spaces. {II}}, volume~97 of {\em Ergebnisse
  der Mathematik und ihrer Grenzgebiete}.
\newblock Springer-Verlag, Berlin-New York, 1979.

\bibitem{Lo69}
G.~Ya. Lozanovskii.
\newblock On some {B}anach lattices.
\newblock {\em Siberian Mathematical Journal}, 10(3):419--431, 1969.

\bibitem{Ma04}
L.~Maligranda.
\newblock Type, cotype and convexity properties of quasi-{B}anach spaces.
\newblock In {\em Banach and function spaces}, pages 83--120. Yokohama Publ.,
  Yokohama, 2004.

\bibitem{MN91}
P.~Meyer-Nieberg.
\newblock {\em Banach lattices}.
\newblock Universitext. Springer-Verlag, Berlin, 1991.

\bibitem{N18}
Z.~Nieraeth.
\newblock {Quantitative estimates and extrapolation for multilinear weight
  classes}.
\newblock arXiv:1808.01985, 2018.

\bibitem{Ru82}
J.~L. Rubio~de Francia.
\newblock Factorization and extrapolation of weights.
\newblock {\em Bull. Amer. Math. Soc. (N.S.)}, 7(2):393--395, 1982.

\bibitem{Ru85}
J.~L. Rubio~de Francia.
\newblock A {L}ittlewood-{P}aley inequality for arbitrary intervals.
\newblock {\em Rev. Mat. Iberoam.}, 1(2):1--14, 1985.

\bibitem{Ru86}
J.~L. Rubio~de Francia.
\newblock Martingale and integral transforms of {B}anach space valued
  functions.
\newblock In {\em Probability and {B}anach spaces ({Z}aragoza, 1985)}, volume
  1221 of {\em Lecture Notes in Math.}, pages 195--222. Springer, Berlin, 1986.

\bibitem{Sc10}
A.~Schep.
\newblock Products and factors of {B}anach function spaces.
\newblock {\em Positivity}, 14(2):301--319, 2010.

\bibitem{SS92}
X.~L. Shi and Q.~Y. Sun.
\newblock Weighted norm inequalities for {B}ochner-{R}iesz operators and
  singular integral operators.
\newblock {\em Proc. Amer. Math. Soc.}, 116(3):665--673, 1992.

\bibitem{St76}
E.~M. Stein.
\newblock Maximal functions. {I}. {S}pherical means.
\newblock {\em Proc. Nat. Acad. Sci. U.S.A.}, 73(7):2174--2175, 1976.

\bibitem{Tr78}
H.~Triebel.
\newblock {\em Interpolation theory, function spaces, differential operators},
  volume~18 of {\em North-Holland Mathematical Library}.
\newblock North-Holland Publishing Co., Amsterdam-New York, 1978.

\bibitem{Xu15}
Q.~Xu.
\newblock {$H^\infty$} functional calculus and maximal inequalities for
  semigroups of contractions on vector-valued {$L_p$}-spaces.
\newblock {\em Int. Math. Res. Not. IMRN}, (14):5715--5732, 2015.

\end{thebibliography}

\end{document}